\documentclass[12pt]{amsart}

\usepackage{enumerate, amsmath, amsthm, amsfonts, amssymb, xy,  mathrsfs, graphicx, paralist, fancyvrb, quiver, ytableau, eucal}
\usepackage[margin=1in]{geometry} 
\usepackage[bookmarks, colorlinks=true, linkcolor=blue, citecolor=blue, urlcolor=blue]{hyperref}
\usepackage[nobysame, alphabetic, citation-order]{amsrefs} 
\AtBeginDocument{%
   \def\MR#1{}
 }

\input xy
\xyoption{all}

\numberwithin{equation}{subsection}
\newtheorem{theorem}[equation]{Theorem}
\newtheorem*{thm}{Theorem}
\newtheorem{proposition}[equation]{Proposition}
\newtheorem{lemma}[equation]{Lemma}
\newtheorem{corollary}[equation]{Corollary}
\newtheorem{conjecture}[equation]{Conjecture}

\theoremstyle{definition}
\newtheorem{rmk}[equation]{Remark}
\newenvironment{remark}[1][]{\begin{rmk}[#1] \pushQED{\qed}}{\popQED \end{rmk}}
\newtheorem{eg}[equation]{Example}
\newenvironment{example}[1][]{\begin{eg}[#1] \pushQED{\qed}}{\popQED \end{eg}}
\newtheorem{defn}[equation]{Definition}
\newenvironment{definition}[1][]{\begin{defn}[#1]\pushQED{\qed}}{\popQED \end{defn}}
\newtheorem{cons}[equation]{Construction}
\newenvironment{construction}[1][]{\begin{cons}[#1]\pushQED{\qed}}{\popQED \end{cons}}

\newenvironment{subeqns}[1][]{\addtocounter{equation}{-1}
\begin{subequations}

}{\end{subequations}}

\newcommand{\cA}{\mathcal{A}}

\newcommand{\rB}{\mathrm{B}}

\newcommand{\bC}{\mathbf{C}}

\newcommand{\rC}{\mathrm{C}}

\newcommand{\rD}{\mathrm{D}}

\newcommand{\bE}{\mathbf{E}}
\newcommand{\cE}{\mathcal{E}}

\newcommand{\bF}{\mathbf{F}}

\newcommand{\rH}{\mathrm{H}}

\newcommand{\bK}{\mathbf{K}}

\newcommand{\rK}{\mathrm{K}}

\newcommand{\cM}{\mathcal{M}}

\newcommand{\bO}{\mathbf{O}}
\newcommand{\cO}{\mathcal{O}}

\newcommand{\cQ}{\mathcal{Q}}

\newcommand{\cR}{\mathcal{R}}

\newcommand{\bS}{\mathbf{S}}

\newcommand{\rS}{\mathrm{S}}

\newcommand{\bU}{\mathbf{U}}
\newcommand{\cU}{\mathcal{U}}

\newcommand{\rU}{\mathrm{U}}

\newcommand{\bZ}{\mathbf{Z}}

\newcommand{\Cliff}{\mathrm{Cliff}}

\newcommand{\fg}{\mathfrak{g}}

\newcommand{\bh}{\mathbf{h}}

\newcommand{\fk}{\mathfrak{k}}

\newcommand{\fn}{\mathfrak{n}}

\newcommand{\fo}{\mathfrak{o}}

\newcommand{\fp}{\mathfrak{p}}

\newcommand{\cat}{\mathcal}

\newcommand{\spo}{\mathfrak{spo}}

\newcommand{\defi}[1]{{\bf\upshape\sffamily #1}}
\renewcommand{\phi}{\varphi}
\renewcommand{\emptyset}{\varnothing}
\newcommand{\eps}{\varepsilon}

\newcommand{\ol}[1]{\overline{#1}}
\newcommand{\ul}[1]{\underline{#1}}

\newcommand{\g}{\fg}
\newcommand{\uU}{\rU}
\newcommand{\gl}{\fgl}
\newcommand{\so}{\fso}
\newcommand{\tor}{\operatorname{Tor}}

\makeatletter
\def\Ddots{\mathinner{\mkern1mu\raise\p@
\vbox{\kern7\p@\hbox{.}}\mkern2mu
\raise4\p@\hbox{.}\mkern2mu\raise7\p@\hbox{.}\mkern1mu}}
\makeatother

\DeclareMathOperator{\coker}{coker}
\renewcommand{\hom}{\operatorname{Hom}}

\DeclareMathOperator{\rank}{rank}

\DeclareMathOperator{\Tor}{Tor}

\DeclareMathOperator{\Spec}{Spec}

\newcommand{\GL}{\mathbf{GL}}

\newcommand{\Sp}{\mathbf{Sp}}

\newcommand{\Gr}{\mathbf{Gr}}

\newcommand{\fgl}{\mathfrak{gl}}
\newcommand{\fso}{\mathfrak{so}}
\newcommand{\fsp}{\mathfrak{sp}}

\newcommand{\op}{\textrm{op}}

\usepackage{longtable}
\usepackage{mathtools}
\usepackage{tikz}
\usetikzlibrary{chains}

\title[Character coincidence between types B and BC]{An infinite-dimensional character coincidence between Lie algebras of type B and BC}

\author{Steven V Sam}
\address{Department of Mathematics, University of California San Diego}
\email{ssam@ucsd.edu}

\author{Keller VandeBogert}
\address{Department of Mathematics, University of Notre Dame, Notre Dame}
\email{kvandebo@nd.edu}

\begin{document}

\maketitle

\begin{abstract}
    We utilize an isomorphism between the character rings of the odd orthogonal group and the orthosymplectic supergroup to understand equivariant positivity properties of the type B quadric hypersurface ring. Our main result establishes a well-behaved functorial construction of Schur modules ``with respect to'' the quadric hypersurface ring, an essential fact used by the authors in previous work to construct pure free resolutions. Our techniques combine ideas from commutative algebra, Lie theory, and algebraic geometry to understand representations of orthosymplectic supergroups and their corresponding type B counterparts. 
\end{abstract}

\tableofcontents

\section{Introduction}

This paper explores connections between Lie theory, combinatorics, and commutative algebra, centering on a certain infinite-dimensional character coincidence between orthosymplectic Lie superalgebras and orthogonal Lie algebras. Our motivation for this construction is to understand the problem of constructing ``quadric" Schur functors and their representation-theoretic properties by taking advantage of a well-known isomorphism between the character rings of the odd-orthogonal Lie group $\bO(2m+1)$ (type B) and the orthosymplectic supergroup $\Sp\bO (2m | 1)$ (type BC). Though this isomorphism has been understood by physicists for quite some time, there are some curious consequences related to equivariant positivity properties of the quadric hypersurface ring that are particularly relevant for the work \cite{sam2024total}. The central construction tying these two character rings together is a variant of the Lie-theoretic construction put forth in \cite{sam2024total}, yielding an infinite-dimensional bridge between type B and BC Lie algebras.

Given any integer $m \geq 1$, the isomorphism of character rings $\text{K} (\bO (2m+1)) \cong {\rm K} (\Sp \bO (2m | 1)) $ implies that for any finite-dimensional $\bO (2m+1)$-representation $W$, there is a corresponding representation of $\Sp \bO (2m|1)$ whose character coincides with that of $W$ and vice versa. For details, see the proof of Corollary~\ref{cor:char-id}. It follows that there are well-defined representations filling in the ``gaps'' in the following table:
\begin{center}
    \begin{tabular}{|c|c|}
  \hline
  $\Sp\bO(2m|1)$  & $\bO(2m+1)$\\  \hline
  irreducible & irreducible \\ \hline
  Schur functors $\bS_\lambda (\bC^{2m|1})$ & ?? \\ \hline
  ?? & Schur functors $\bS_\lambda (\bC^{2m+1})$ \\ \hline
  \end{tabular}

\end{center}
It is often easier/more natural to construct these representations directly on one side of the table as opposed to the other. That being said, the corresponding representations on the ``other side'' may exhibit structure or properties that motivate a direct construction. For instance, the $\bO (2m+1)$-representation whose character is equal to the $\Sp \bO (2m |1)$-character of the Schur functor $\bS_{\lambda} (\bC^{2m|1})$, whatever it is, \emph{must} satisfy the Jacobi--Trudi identity ``with respect to'' the quadric hypersurface ring corresponding to a maximal rank quadric. This observation immediately yields $\bO (2m+1)$-equivariant total positivity of the (odd) quadric hypersurface ring, but as motivated by \cite{sam2024total}, a more ``global" approach to the problem of equivariant positivity lends itself to constructions that behave much better for algebraic purposes. Our main theorem shows that this proposed positivity program applies directly in the type B setting:

\begin{thm}
    Let $V$ be a $(2m+1)$-dimensional orthogonal space and set $A \coloneqq S^\bullet (V) / (q)$ where $q \in S^2 (V)$ denotes the induced quadric. Then, for every $n > 0$ there exists an $\bO(V) \times \GL_n$-representation $Z_{V, \bC^n}$ with ``good functorial properties'' such that
    \[
      (Z_{V, \bC^n})_\lambda \cong A_{\lambda_1} \otimes \cdots \otimes A_{\lambda_n},
    \]
    where $\lambda \in \bZ^n$ is any weight and $(-)_\lambda$ denotes the $\GL_n$-weight space.

    The $\bO(2m+1) \times \GL_n$-character of $Z_{V,\bC^n}$ coincides with the $\Sp \bO (2m|1) \times \GL_n$-character of $S^\bullet (\bC^{2m|1} \otimes \bC^n)$. 
\end{thm}

In the terminology of \cite{sam2024total}, this endows the odd quadric hypersurface ring with a \defi{Jacobi--Trudi structure}. In the case of the even quadric hypersurface ring (i.e., the type D case), there are some failures of equivariant positivity that turn out to be related to character coincidences stemming from work of Proctor \cite{proctor1988odd} on the ``odd" symplectic group $\Sp (2m+1)$; see Remark \ref{rk:oddSymplecticRk}. This means that the Theorem fails in the type D setting, yielding a surprising distinction in the type B versus type D positivity behavior of the quadric hypersurface ring. 

If we ignore the ``good functorial properties'' clause, then our theorem is a formal consequence of the above discussion regarding characters. The functorial properties will be used in \cite{sam2024total} and not here, so rather than spell them out explicitly in this introduction, we will instead outline the construction of $Z_{V,\bC^n}$. As before, let $V$ denote an odd-dimensional orthogonal space with induced quadric $q \in S^2 (V)$ of full rank. Let $\beta$ be the symmetric bilinear form on $V$ associated to $q$. We consider the $\bZ$-graded Lie algebra concentrated in degrees 1 and 2
\[
  \g_{V,\bC^n} = (V \otimes \bC^n) \oplus \bigwedge^2 \bC^n
\]
with Lie bracket on degree 1 elements given by $[v \otimes e, v' \otimes e'] = \beta(v,v') e \wedge e'$. The form $\beta$ gives an identification $V \cong V^*$, and hence there is a canonical inclusion
\[
  S^2(\bC^n) \subset S^2(V \otimes \bC^n) \subset \uU(\g_{V,\bC^n}),
\]
where $\uU$ denotes universal enveloping algebra. Let $I$ denote the left ideal generated by the subspace $S^2(\bC^n)$ and define
\[
  Z_{V,\bC^n} \coloneqq \uU (\fg_{V,\bC^n}) / I.
\]
We note that $I$ is not a 2-sided ideal in general, so that $Z_{V, \bC^n}$ does not have any natural algebra structure. Nonetheless, from this construction, it inherits the structure of a left $\rU(\fg_{V,\bC^n})$-module, and this leads to the ``good functorial properties'' alluded to above.

 There is a further significance to this construction which is crucial to our approach. The Lie algebra $\g_{V,\bC^n}$ is the nilpotent radical of a parabolic subalgebra of $\mathfrak{so}(2m+2n+1)$ and $\uU(\g_{V,\bC^n})$ is a parabolic Verma module (technically, for what follows we have to twist this by a non-trivial character). Our main result here is that the left ideal $I$ is in fact an $\mathfrak{so}(2m+2n+1)$-subrepresentation and the quotient $\uU(\g_{V,\bC^n})/I$ is an irreducible $\mathfrak{so}(2m+2n+1)$-representation. This family of representations interpolates between spinor representations and the quadric hypersurface, see Example~\ref{ex:interpolate}.

 Proving this result is, in a sense, the main bulk of this paper, so we now outline how this is done. First, we show that $I$ is the image of a standard homomorphism between parabolic Verma modules using work of Boe \cite{boe1985homomorphisms}. Second, by using the PBW degeneration, we get an upper bound on the character of $Z_{V,\bC^n}$ from the character of the commutative algebra quotient
 \[
   S^\bullet(V \otimes \bC^n) / (S^2 \bC^n) \otimes S^\bullet(\bigwedge^2 \bC^n).
 \]
 We analyze this degeneration by considering the action of $\bO(V) \times \fso(2n)$; this representation is a direct sum of finite-dimensional $\bO(V)$-representations, and each multiplicity space is a parabolic Verma module for $\fso(2n)$. This immediately gives a lower bound for the character of $Z_{V,\bC^n}$ where the parabolic Verma modules are replaced by their unique irreducible quotients. Ultimately we prove that $Z_{V,\bC^n}$ agrees with this lower bound.

 The infinite-dimensional $\fso(2n)$-representations that appear in the decomposition are very similar to those appearing in classical Howe dual pairs, but the weights do not exactly match up. Instead, we establish that there is a Howe dual pair (which we believe is new) that involves the orthosymplectic Lie supergroup $\Sp\bO(2m|1)$ in which all of these representations show up. Our main contribution here is to calculate the Lie algebra homology of these infinite-dimensional representations. This is very closely related to the notion of Littlewood complexes, as studied in \cite{sam2015littlewood, sam2013homology}. Remarkably, these homology calculations allow us to conclude that the only $\fso(2m+2n+1)$-representation whose character satisfies the two bounds above must agree with the lower bound. We remark that in order to develop everything in a logical fashion, the order of the sections of this article is reversed from the order of the topics in this explanation.

 We also remark that the homological calculations above allow us to compute the Tor groups for a class of non-reduced determinantal ideals which behave in many ways like the classically well-studied examples (determinantal ideals, Pfaffian ideals, and determinantal ideals of symmetric matrices). We highlight this special case in \S\ref{ss:ideal}.
 
Returning to the Jacobi--Trudi structure, the Schur modules for the algebra $A$ that we construct are isotypic components of $\uU(\g_{V,\bC^n})/I$ under the action of $\GL_n$. Since the $\bO(V)$ action commutes with $\GL_n$, these Schur modules are $\bO(V)$-representations, but they are practically never irreducible. Invoking a form of Howe duality, we conclude that they are irreducible modules for the subalgebra of $\GL_n$-invariants in $\uU(\mathfrak{so}(2m+2n+1))$. The nature of this algebra is still mysterious to us (see \S\ref{subsec:quadricAsIrred}).

We also want to mention a tentative connection with the theory of minimal affinizations, a certain class of finite-dimensional representations over quantum loop algebras associated with a semisimple Lie algebra $\g$. In their approach to studying them, Chari and Greenstein \cite{chari-greenstein} consider a class of modules over the current algebra $\g \otimes \bC[t]$ and conjectured a formula for their character, which was proven in \cite{JTdet-minaff}. Surprisingly, for $\g=\mathfrak{so}(2m+1)$, these formulas coincide identically with the characters of the Schur modules constructed from the Jacobi--Trudi structure above. We suspect that there is some relation between the current algebra and the invariant subalgebra $\uU(\mathfrak{so}(2m+2n+1))^{\GL_n}$.

Finally, we have only focused on the missing representations on the $\bO(2m+1)$ side of the table above, but it is reasonable to expect that the missing representations on the $\Sp\bO(2m|1)$ side of the table can be treated analogously. We collect together some basic facts and conjectures about this case in \S\ref{ss:orthosymp-version}.

\medskip 
\noindent {\bf Organization. }
We end with a brief explanation of this paper's organization. As explained above, the summary of this paper's main arguments above is in reverse order from how the sections are actually laid out. In \S \ref{sec:prelims} we recall some background related to Frobenius coordinates and Bott's algorithm in this setting. In \S \ref{sec:SpoLittlewood} we formulate an analogue of the Littlewood complexes introduced in \cite{sam2015littlewood} for the orthosymplectic group $\Sp \bO (2m|1)$, including some character computations relating representations of $\Sp \bO (2m|1)$ with $\bO (2m+1)$, and in \S \ref{sec:homologyOfLittlewood} we compute the homology of these complexes using a method that is similar to that used in \cite{sam2015littlewood}. In \S \ref{sec:howe} we prove that there is a Howe dual pair $(\Sp \bO (2m|1) , \so (2n))$ with respect to the representation $S^\bullet (\bC^{2m|1} \otimes \bC^n)$; this fact will be crucial for establishing that $Z_{V,\bC^n}$ and $S^\bullet (\bC^{2m|1} \otimes \bC^n)$ have the same character via the isomorphism of character rings $\rK (\bO (2m+1)) \cong \rK(\Sp \bO (2m | 1)) $. In \S \ref{sec:generalOddRank} we prove the main results of the paper relating to the representation $Z_{V,E}$ and its character. 

\section{Preliminaries}\label{sec:prelims}

\subsection{Partitions and Frobenius coordinates}

A \defi{partition} is a tuple of nonnegative integers $\lambda = (\lambda_1 , \dots , \lambda_n)$ with $\lambda_1 \geq \cdots \geq \lambda_n$. Its \defi{length}, denoted $\ell(\lambda)$, is the number of nonzero $\lambda_i$. A \defi{Young diagram} of shape $\lambda$ is a left-justified diagram whose $i$th row has length $\lambda_i$. Given a partition $\lambda$, its \defi{transpose} $\lambda^T$ is the partition obtained by reading off the column lengths of the Young diagram corresponding to $\lambda$. In a formula: $\lambda^T_i = \# \{j \mid \lambda_j \ge i\}$. We often treat partitions as vectors, e.g., $2\lambda = (2\lambda_1,2\lambda_2,\dots)$.

  The notation $(a^b)$ will denote a sequence of $b$ copies of $a$.

\begin{definition}
    Let $\lambda = (\lambda_1 , \dots , \lambda_n)$ be a partition.  The \defi{rank} of $\lambda$, denoted $\rank(\lambda)$, is the largest integer $d$ such that $\lambda_d \ge d$. Then the \defi{Frobenius coordinates} $(a_1 , \dots , a_d \mid b_1 , \dots , b_d)$ of $\lambda$ (here $d$ is the rank of $\lambda$) are given by reading off arm and leg lengths along the main diagonal of the Young diagram representing $\lambda$. More precisely:
    \[
      a_i \coloneq \lambda_i -i, \qquad
    b_i \coloneq \lambda_i^T - i. \qedhere
      \]
  \end{definition}

\begin{definition}\label{def:theQsets}
    Given integers $i \in \bZ$ and $d \geq 0$, define the set
    \[
      Q_i^d \coloneq \{ (a_1 , \dots , a_d \mid b_1 , \dots , b_d) \mid a_j - b_j = i \quad \text{for} \ j = 1 , \dots , d \}.
    \]
    Likewise, define the sets
    $$Q_i \coloneq \bigcup_{d \geq 0} Q_i^d, \quad Q_i^{\rm ev} \coloneq \bigcup_{\substack{d \geq 0 \\ d \ \text{even}}} Q_i^d, \quad Q_i^{\rm od} \coloneq \bigcup_{\substack{d \geq 0 \\ d \ \text{odd}}} Q_i^d.$$
    Given an element $(\underline{a} \mid \underline{b}) \in Q_i^d$, define the \defi{degree} by
    \[
      \deg(\underline{a} \mid \underline{b}) \coloneq a_1 + \cdots + a_d . \qedhere
    \]
  \end{definition}

    For background on Schur functors, see \cite[\S 2]{weyman2003}, but note that the indexing used there is transposed from standard conventions, i.e., the Schur functor indexed by $\lambda$ there is what we would call $\bS_{\lambda^T}$.
  
  \begin{remark}\label{rk:FrobAbuse}
    When $i$ is fixed, given a partition $\lambda$ with Frobenius coordinates $(\ul{a} \mid \ul{b}) \in Q_i$, we will sometimes abuse notation and write
    $$\bS_{\underline{a}} \coloneq \bS_{(\underline{a} | \underline{b})} \coloneq \bS_\lambda$$
    to denote the Schur functor corresponding to $\lambda$. The convention of sometimes ignoring the term $\underline{b}$ should cause no confusion, since $\underline{b}$ is determined by $\underline{a}$ whenever $(\underline{a} \mid \underline{b}) \in Q_i$.
  \end{remark}

  For the following, see \cite[Propositions 2.3.8, 2.3.9]{weyman2003}. We use the notation $S^d$ to denote the $d$th symmetric power functor, and $S^\bullet = \bigoplus_{d \ge 0} S^d$. Similarly, $\bigwedge^d$ is the $d$th exterior power functor and $\bigwedge^\bullet = \bigoplus_{d \ge 0} \bigwedge^d$.

\begin{proposition} \label{prop:plethysms}
    We have the following direct sum decompositions:
        $$ S^\bullet ( S^2 ) = \bigoplus_{\lambda} \bS_{2  \lambda}, \qquad S^\bullet ( \bigwedge^2 ) = \bigoplus_{\lambda} \bS_{(2 \lambda)^T}, \qquad 
        \bigwedge^\bullet (S^2) = \bigoplus_{\lambda \in Q_{1}} \bS_\lambda, \qquad  
        \bigwedge^\bullet (\bigwedge^2 ) = \bigoplus_{\lambda \in Q_{-1}} \bS_\lambda$$
        where the left hand side of each equation is a composition of exterior and/or symmetric powers and the first two sums range over all partitions $\lambda$.
\end{proposition}

  \subsection{Borel--Weil--Bott with Frobenius coordinates}\label{subsec:BWBFrob}

Given a vector space $F$ of dimension $n$ and an integer $0 \le k \le n$, we let $\Gr(k,F) = \Gr(k,n)$ denote the Grassmannian which parametrizes $k$-dimensional subspaces of $F$. It comes equipped with a tautological exact sequence of vector bundles
\[
  0 \to \cR \to F \times \Gr(k,F) \to \cQ \to 0
\]
where $\cR = \{(x,W) \mid x \in W\}$ has rank $k$ and $\cQ$ has rank $n-k$. We recall Bott's theorem for the Grassmannian (see \cite[Corollary 4.1.9]{weyman2003}):

\begin{theorem}[Borel--Weil--Bott] \label{thm:bott}
  Let $\lambda,\mu$ be partitions with $\ell(\lambda) \le k$ and $\ell(\mu) \le n-k$ and set $\cE = \bS_\lambda(\cR) \otimes \bS_\mu(\cQ)$. Define
  \[
    \alpha = (\mu_1,\dots,\mu_{n-k}, \lambda_1,\dots,\lambda_k)
  \]
  and $\rho = (n-1,n-2,\dots,1,0)$.
  \begin{enumerate}
  \item If $\alpha + \rho$ has any repetitions, then all cohomology of $\cE$ vanishes.
    
  \item Otherwise, there is a unique permutation $\sigma$ such that $\sigma(\alpha+\rho)$ is decreasing. Then
    \[
      \rH^i(\Gr(k,F), \cE) = \begin{cases} \bS_{\sigma(\alpha+\rho)-\rho}(F) & \text{if $i=\ell(w)$}\\ 0 & \text{if $i \ne \ell(w)$} \end{cases}.
    \]
  \end{enumerate}
\end{theorem}

\begin{corollary}\label{cor:easyFrobCase}
    Let $(\underline{a} \mid \underline{b}) = (a_1 ,\dots , a_d \mid b_1 , \dots ,b_d)$ denote the Frobenius coordinates of a partition $\lambda = (\lambda_1 , \dots , \lambda_k)$. If $a_d < n-k$, then all cohomology of $\bS_{(\ul{a} | \ul{b})}(\cR)$ vanishes. Otherwise, we have
    \[
      \rH^i (\Gr (k,F) , \bS_{(\underline{a} \mid \underline{b})} (\cR) ) = \begin{cases} 
        \bS_{(\underline{a} - (n-k)^d \mid \underline{b} + (n-k)^d)} (F) & \text{if $i = d(n-k)$} \\
        0 & \text{if $i \ne d(n-k)$}
      \end{cases}.
    \]
  \end{corollary}
  
  \begin{proof}
    First, we have
    \[
      \lambda  = (a_1+1 , a_2+2, \dots , a_d+d, \nu_1 , \dots , \nu_{k-d})
    \]
    for some partition $\nu \coloneq (\nu_1 , \dots , \nu_{k-d})$ with $\nu_1 \leq d$.  In the notation of Bott's theorem above, 
    \[
      \alpha + \rho = (n-1 , n-2 , \dots , k , a_1 + k , a_2 + k, \dots , a_d + k, \nu_1 +k-d-1 , \dots, \nu_{k-d}).
    \]
    Since $\nu_1 \le d$, we have $a_d+k > \nu_1 + k  - d -1$, and hence this sequence has repetitions if and only if $a_d < n-k$. If $a_d \ge n-k$, then the sorted sequence is
    \[
      (a_1 + k , a_2 + k, \dots , a_d + k, n-1 , n-2 , \dots , k , \nu_1 +k-d-1 , \dots, \nu_{k-d})
    \]
    and the permutation $\sigma$ has length $d(n-k)$. Finally, subtracting $\rho$ gives
    \[
      (a_1 + k-n+1, \dots , a_d + k-n+d, \underbrace{d, \dots, d}_{n-k}, \nu_1 , \dots, \nu_{k-d})
    \]
    and this partition has Frobenius coordinates $(a_1-(n-k), \dots, a_d - (n-k) \mid b_1 + (n-k), \dots, b_d + (n-k))$.
  \end{proof}
  
\section{Littlewood complexes for $\Sp\bO(2m|1)$}\label{sec:SpoLittlewood}

\subsection{Construction}\label{sec:LittlewoodCons}

In this section, we introduce an orthosymplectic analogue of the classical Littlewood complexes, as studied in \cite{sam2013homology, sam2017infinite}. These complexes provide resolutions of specific representations of \(\Sp\bO(2m|1)\), whose characters will play a central role in later computations. We will use repeatedly that the finite-dimensional representations of $\Sp\bO(2m|1)$ are finite-dimensional (see \cite[Corollary 2.33]{cheng2012dualities}).
  
    Let $E$ be any vector space and $U \coloneq \bC^{2m | 1}$ a superspace (so $U_0 = \bC^{2m}$, $U_1=\bC$) equipped with an orthosymplectic form $\omega \in \bigwedge^2 U_0 \oplus S^2 U_1$. The superalgebra $S^\bullet (E \otimes U)$ comes equipped with an $S^\bullet (\bigwedge^2 E)$-module structure via the inclusion
    \[
      \bigwedge^2 E \subset \bigwedge^2 E \otimes \left( \bigwedge^2 U_0 \oplus S^2 U_1 \right)  \subset S^\bullet (E \otimes U).
    \]
    Consider the $\GL(E)$-action on the Koszul complex
        \[
      \bK_{E,U} = S^\bullet (E \otimes U) \otimes_{S^\bullet (\bigwedge^2 E)} \bigwedge^\bullet (\bigwedge^2 E).
    \]
    We let $L^\lambda_\bullet(U)$ denote the $\bS_\lambda(E)$-isotypic component of $\bK_{E,U}$ and call it the \defi{Littlewood complex} for $\Sp\bO(U) \cong \Sp\bO (2m |1)$; it is independent of the choice of $E$ as long as $\dim E \ge \ell(\lambda)$ (so that $\bS_\lambda (E) \ne 0$). More concretely, the complex $L^\lambda_\bullet (U)$ has terms described explicitly via 
    \[
      L_i^\lambda (U) = \bigoplus_{\substack{\mu \in Q_{-1} \\ |\mu| = 2i}} \bS_{\lambda / \mu} (U). \qedhere
  \]
By semisimplicity, $\bK_{E,U}$  has a $\GL (E) \times \Sp\bO(U)$-equivariant decomposition
    \[
      \bK_{E,U} \cong \bigoplus_{\lambda} \bS_\lambda (E) \otimes  L^\lambda_\bullet (U).
    \]

\begin{proposition}\label{prop:superRegSeq}
  If $\dim E \le m$, then the Koszul complex $\bK_{E,U}$ is acyclic. In particular, $L^\lambda_\bullet(U)$ is acyclic whenever $\ell(\lambda) \le m$ and we have an isomorphism of representations
  \[
    S^\bullet(E \otimes U) \cong (S^\bullet(E \otimes U) / (\bigwedge^2 E) ) \otimes S^\bullet(\bigwedge^2 E).
  \]
\end{proposition}

As we will see in Corollary~\ref{cor:adm-acyclic}, $L^\lambda_\bullet(U)$ remains acyclic under the weaker assumption $\lambda_1^T + \lambda_2^T \le 2m+1$.

\begin{proof}
  Let $I$ be the ideal generated by $\bigwedge^2 E$. We consider a 1-parameter family where all odd variables in $S^\bullet(E \otimes U)$ are scaled by $t$. The $t=0$ limit of $I$ contains the ideal generated by the copy of $\bigwedge^2 E$ induced by the inclusion
  \[
    \bigwedge^2 E \subset \bigwedge^2 E \otimes \bigwedge^2 U_0 \subset S^2(E \otimes U_0).
  \]
  This ideal (in $S^\bullet(E \otimes U_0)$) is known to be a complete intersection (for example, by \cite[Lemma 3.3]{sam2013homology}), which implies the result (see for instance, the proof of \cite[Proposition 4.1]{bwfact}).
\end{proof}

We define
\[
  \bS_{[\lambda]}(U) = \rH_0(L^\lambda_\bullet(U)).
\]
When $\ell(\lambda) \le m$, the previous result combined with the plethysms of Proposition~\ref{prop:plethysms} gives the following two character identities:
\begin{subequations}
  \begin{align} 
  [\bS_{[\lambda]}(U)] &= [L^\lambda_\bullet(U)] = \sum_{\mu \in Q_{-1}} (-1)^{|\mu|/2} [\bS_{\lambda / \mu} (U)], \label{eqn:finite-1} \\
  [\bS_{\lambda}(U)] &= \sum_{\mu, \nu} c^{\lambda}_{\mu, (2\nu)^T} [\bS_{[\mu]}(U)]. \label{eqn:finite-2}
\end{align}
\end{subequations}

Let $[L^\lambda_\bullet(U)]$ denote the equivariant Euler characteristic of $L^\lambda_\bullet(U)$. For general $\lambda$ we have the following determinantal formula.

\begin{proposition} \label{prop:euler-det}
  Let $h_i = [S^i(U)]$ and pick $r \ge \ell(\lambda)$. Then
  \[
    [L^{\lambda}_\bullet(U)] = \frac12 \det(h_{\lambda_i - i + j} + h_{\lambda_i - i - j + 2})_{i,j=1}^r.
  \]
\end{proposition}

\begin{proof}
  When $j=1$, we have $h_{\lambda_i-i+j}=h_{\lambda_i-i-j+2}$, so we can instead rewrite the determinant so that the first column is simply $h_{\lambda_i-i+1}$ and omit the $\frac12$ factor. Hence by multilinearity, we can expand the determinant as a sum of $2^{r-1}$ determinants: 
  $$\frac12 \det(h_{\lambda_i - i + j} + h_{\lambda_i - i - j + 2})_{i,j=1}^r = \sum_{I \subseteq \{ 2 , \dots , r \} } \det ( \bh | M(I)),$$
  where $\bh = \begin{pmatrix} h_{\lambda_1} \\ h_{\lambda_2 - 1} \\ \vdots \\ h_{\lambda_r - r+1} \end{pmatrix}$ and $M(I)$ is an $r \times (r-1)$ matrix with entries (the columns are indexed by $2,\dots,r$):
  \[
    M(I)_{i,j} = \begin{cases}
      h_{\lambda_i - i +j} & \text{if} \ j \notin I, \\
      h_{\lambda_i - i - j + 2} & \text{if} \ j \in I.
    \end{cases}
  \]
  For a fixed indexing set $I = \{ j_1 , \dots , j_k \} \subseteq \{ 2 , \dots , r \}$ define the vector $J=(J_1,\dots,J_r)$ with
  \[
    J_i = \begin{cases} 0 & \text{if $i \notin I$} \\ 2i-2 & \text{else}. \end{cases}
  \]
  Define the partition $\mu \in Q_{-1}$ whose Frobenius coordinates are given by
  \[
    \mu =  (j_k-2, j_{k-1}-2 , \dots , j_1-2 \mid j_k-1 , j_{k-1}-1, \dots , j_1-1).
  \]
  In this proof we will use the dotted action of the symmetric group on $\bZ^r$ by $\sigma \bullet \alpha = \sigma(\alpha + \rho) - \rho$ where $\rho = (r-1,r-2,\dots,1,0)$.
  
  We claim that there exists a permutation $\sigma$ of $\{1,\dots,j_k\}$ with $\ell(\sigma)=|\mu|/2$ such that $\sigma \bullet J = \mu$. We argue by induction on $|I|$: when $|I| = 0$, there is nothing to show. For $|I| > 0$ write $J = (J', 2 j_k-2, 0^{r-1-\ell(J')})$ for some vector $J'$ (for simplicity we will ignore the tail of $0$'s on $J$); the cycle permutation $w' = (12 \dots j_k)$ has length $j_k-1$ and $w' \bullet J = (j_k-1 ,J') + (0, 1^{j_k-1})$. By induction, there exists a permutation $w''$ of $\{1,\dots,j_{k-1}\}$ such that $w'' \bullet J' = \mu'$ where $\mu'$ has Frobenius coordinates $(j_{k-1}-2, \dots , j_1-2 \mid j_{k-1}-1 , \dots , j_1-1)$.   Let $w'''$ be the same permutation of $\{2,\dots,j_{k-1}+1\}$ obtained by increasing the value of every entry by 1.

  Set $\sigma = w''' \cdot w'$, which is a permutation of $\{1,\dots,j_k\}$. Then $\sigma \bullet J = (j_k-1 , \mu' + 1^{j_k-1}) = \mu$, and moreover
  \[
    \ell(\sigma) = \ell(w') + \ell(w''') = j_k-1 + |\mu'|/2 = |\mu|/2.
  \]

  Finally, $\sigma$ reorders the vector $J$ into descending order and thus by applying $\sigma$ to the columns of $(\bh|M(I))$, we get
  \[
    \det (\bh | M(I)) = (-1)^{|\mu|/2} \det ( h_{\lambda_i - \mu_j - i +j} )_{i,j=1}^r
  \]
  The determinant on the right side computes $(-1)^{|\mu|/2} [ \bS_{\lambda / \mu} (U)]$ by the Jacobi--Trudi identity \cite[\S 7.16]{stanley}. Taking the sum over each $J \subseteq \{ 2 , \dots , r \}$ will range through all $\mu \in Q_{-1}$, in which case the desired character equality follows. 
\end{proof}

\begin{corollary} \label{cor:char-id}
  If $\ell(\lambda) \le m$, then $\bS_{[\lambda]}(U)$ is an irreducible $\Sp\bO(U)$-representation of highest weight $\lambda$.
\end{corollary}

As will follow from Remark~\ref{rmk:admissible}, the above assumption can be relaxed to $\lambda_1^T + \lambda_2^T \le 2m+1$.

\begin{proof}
    First, the character of the irreducible $\spo(2m|1)$-representation of highest weight $\lambda$ is computed by the Weyl character formula (see \cite[Theorem 2.35]{cheng2012dualities}) where $W$ is the type $\rB\rC_m$ Weyl group:
    \begingroup\allowdisplaybreaks
    \begin{align*}
        \frac{\prod_{i=1}^m (1 + x_i^{-1}) \cdot \left( \sum_{w \in W} (-1)^{\ell (w)} x^{w \bullet \lambda} \right)}{\prod_{i=1}^m (1+x_i^{-2}) \cdot \prod_{1 \leq i < j \leq m} x_i^{-1} x_j^{-1} (x_i + x_i^{-1} - x_j - x_j^{-1})} \\
        = \frac{  \sum_{w \in W} (-1)^{\ell (w)} x^{w \bullet \lambda}}{\prod_{i=1}^m (1+x_i^{-1}) \cdot \prod_{1 \leq i < j \leq m} x_i^{-1} x_j^{-1} (x_i + x_i^{-1} - x_j - x_j^{-1})}. 
    \end{align*}
    \endgroup
    This last expression is the Weyl character formula for the character of the irreducible $\so(2m+1)$-representation of highest weight $\lambda$ \cite[\S 24.1]{fulton2013representation}. 

    Second, by \cite[Theorem 1.3.2]{koiketerada}, the character of the irreducible $\so(2m+1)$-representation of highest weight $\lambda$ is given by
    \[
      \det([S^{\lambda_i - i + j}(\bC^{2m+1})] - [S^{\lambda_i - i - j}(\bC^{2m+1})])_{i,j=1}^m.
    \]
    where $[S^d(\bC^{2m+1})]$ is the sum of all $d$-fold products of elements in the set
    \[
      \{x_1,\dots,x_m,x_1^{-1},\dots,x_m^{-1},1\}.
    \]
    Define $h'_d$ to be the sum of the $d$-fold products of the elements in  $\{x_1,\dots,x_m,x_1^{-1},\dots,x_m^{-1}\}$, which we can think of as $[S^d(\bC^{2m})]$. Then $[S^d(\bC^{2m+1})] = \sum_{i=0}^d h'_i$. On the other hand, using \eqref{eqn:finite-1} and Proposition~\ref{prop:euler-det}, we have
    \[
      [\bS_{[\lambda]}(U)] = \frac12 \det ( [S^{\lambda_i-i+j}(U)] + [S^{\lambda_i-i-j+2}(U)])_{i,j=1}^m.
    \]
    Also, we have $[S^d(U)] = h'_d + h'_{d-1}$. Now we can transform the second determinant into the first via column operations: for each $i \le m-2$, add column $i$ to column $i+2$ (starting from $i=1$ and going in order). In particular, the characters agree, so that $[\bS_{[\lambda]}(U)]$ is the character of an irreducible representation, which means that $\bS_{[\lambda]}(U)$ is itself irreducible.
  \end{proof}

  \begin{remark} \label{rmk:same-euler}
    The second determinant in the above proof can be identified with the Euler characteristic of the Littlewood complex for $\bO(2m+1)$, as defined in \cite[\S 4.2]{sam2013homology}. The proof is similar to the proof for Proposition~\ref{prop:euler-det}. So the Littlewood complexes for $\Sp\bO(2m|1)$ and $\bO(2m+1)$ have the same Euler characteristic (when indexed by the same partition $\lambda$) without any restriction on $\ell(\lambda)$.
  \end{remark}
  
  Finally, we translate these results to characters of the stable limit $\Sp\bO(\bU)$ where $\bU=\bC^{2\infty|1}$ in the sense of \cite{sam2017infinite}. The ring $\rK(\Sp\bO(\bU))$ is the Grothendieck group of a suitable category of tensor representations which is introduced there. For our purposes, it is enough to take $\rK(\Sp\bO(\bU))$ to be the ring of symmetric functions (which we will identify with the character ring of $\GL(\bU)$).

  Since the identities \eqref{eqn:finite-1} are valid whenever $m \ge \ell(\lambda)$, we define
\begin{subequations}
  \begin{align} 
  [\bS_{[\lambda]}(\bU)] &= \sum_{\mu \in Q_{-1}} (-1)^{|\mu|/2} [\bS_{\lambda / \mu} (\bU)]. \label{eqn:infinite-1}
  \end{align}
Now using the identities \eqref{eqn:finite-2}, we conclude that
\begin{align}
  [\bS_{\lambda}(\bU)] &= \sum_{\mu, \nu} c^{\lambda}_{\mu, (2\nu)^T} [\bS_{[\mu]}(\bU)]. \label{eqn:infinite-2}
\end{align}
\end{subequations}
Note that if $c^\lambda_{\mu,(2\nu)^T} \ne 0$, then $|\mu|\le |\lambda|$, and when $\lambda=\mu$, this coefficient is 1 for $\nu=\emptyset$ and is 0 otherwise. It follows that the set of $[\bS_{[\lambda]}(\bU)]$ forms a basis for $\rK(\Sp\bO(\bU))$. Furthermore, the classes $[S^d(\bU)]$, for $d>0$, form an algebraically independent set of ring generators (this is just a statement about the complete homogeneous symmetric functions).

We will also need to make use of the stable limit $\bO(\infty)$ as studied in \cite{stability-patterns}.

Define a ring homomorphism from $\rK(\Sp\bO(\bU))$ to the character ring $\rK(\bO(\infty))$ which sends $[S^d(\bU)]$ to $[S^d(\bC^{\infty})] - [S^{d-2}(\bC^{\infty})]$. From the proof of Corollary~\ref{cor:char-id}, we deduce that $[\bS_{[\lambda]}(\bU)] \mapsto [\bS_{[\lambda]}(\bC^{\infty})]$ where the latter is the character of the corresponding irreducible representation of $\bO(\infty)$.

Furthermore, we have specialization maps
\[
  \pi_{\Sp\bO(2m|1)} \colon \rK(\Sp\bO(\bU)) \to \rK(\Sp\bO(2m|1)), \qquad [S^d(\bU)] \mapsto [S^d(\bC^{2m|1})]
\]
and
\[
  \pi_{\bO(2m+1)} \colon \rK(\bO(\infty)) \to \rK(\bO(2m+1)), \qquad [S^d(\bC^{\infty})] \mapsto [S^d(\bC^{2m+1})].
\]

\begin{proposition} \label{prop:comm-square}
  With the notation above, we have a commutative square
  \[
    \xymatrix{ \rK(\Sp\bO(2\infty|1)) \ar[r]^-\cong \ar[d]_-{\pi_{\Sp\bO(2m|1)}} & \rK(\bO(\infty)) \ar[d]^-{\pi_{\bO(2m+1)}} \\
      \rK(\Sp\bO(2m|1)) \ar[r]^-\cong & \rK(\bO(2m+1)) }
  \]
  which preserves indexing of irreducible representations by partitions.
\end{proposition}

\begin{proof}
  This follows from the above discussion since the elements $[S^d(\bU)]$ are ring generators for $\rK(\Sp\bO(2\infty|1))$.
\end{proof}

\begin{definition}
  Fix $m$.   A partition $\lambda$ is \defi{admissible} if $\lambda_1^T + \lambda_2^T \leq 2m+1$. In that case, we can define a partition $\lambda^\sigma$ by replacing the first column length of $\lambda$ with $2m +1 - \lambda_1^T$. Then $\lambda^\sigma$ is also admissible and exactly one of $\lambda$ and $\lambda^\sigma$ has length at most $m$; this particular partition will be denoted $\ol{\lambda}$.
\end{definition}

\begin{remark} \label{rmk:admissible}
  For $\bO(2m+1)$, we have $\bS_{[\lambda]}(\bC^{2m+1}) \cong \det \otimes \bS_{[\lambda^\sigma]}(\bC^{2m+1})$ where $\det$ is the determinant character \cite[Exercise 19.23]{fulton2013representation}. In particular, the determinant character is $\bS_{[1^{2m+1}]}(\bC^{2m+1})$. From the explicit description of the Littlewood complex, we have that $\bS_{[1^{2m+1}]}(\bC^{2m|1})$ is the determinant character of $\Sp\bO(2m|1)$ (specifically, this can be thought of as the nontrivial character of the even subgroup $\bO(1) \cong \bZ/2$).

  Hence from the isomorphism above, if $\lambda$ is admissible, we conclude that $\bS_{[\lambda]}(\bC^{2m|1}) \cong \det \otimes \bS_{[\lambda^\sigma]}(\bC^{2m|1})$ and that the Euler characteristic of $L^\lambda_\bullet(U)$ is $[\bS_\lambda(U)]$ (we will see in Corollary~\ref{cor:adm-acyclic} that $L^\lambda_\bullet(U)$ is acyclic).
\end{remark}

\begin{remark}\label{rk:oddSymplecticRk}
     While the even-orthogonal (type D) quadric hypersurface ring may not be equivariantly positive, there is a similar character correspondence that follows from work of Proctor \cite{proctor1988odd}. More precisely, let $\Sp \bO (2m+1 | 1)$ denote the odd orthosymplectic supergroup, which contains only an $m$-dimensional torus. Then there is an isomorphism of character rings
    $${\rm K} (\Sp \bO (2m+1 |1) ) \cong {\rm K} (\bO (2m+2))|_{x_{m+1} = 1},$$
    where the notation ${\rm K} (\bO (2m+2))|_{x_{m+1} = 1}$ denotes the character ring ${\rm K} (\bO (2m+2))$ but with the last variable $x_{m+1}$ set equal to $1$. Under this isomorphism, the virtual character of the type D quadric Schur module $[\bS^A_\lambda]$ (after setting the last variable equal to $1$) corresponds to the $\Sp\bO(2m+1|1)$ character of the classical Schur module $\bS_\lambda (\bC^{2m+1|1})$.  Thus the even (type D) quadric hypersurface ring is totally positive only after restricting to a torus of submaximal dimension in $\bO (2m+2)$. 
\end{remark}

\subsection{Modification rules}

Given the commutative diagram in Proposition~\ref{prop:comm-square}, we can use the existing combinatorial rules from \cite[\S 4.4]{sam2013homology} for computing $\pi_{\bO(2m+1)}$ to compute $\pi_{\Sp\bO(2m|1)}$. 
In this section, we recall two of these rules, known as ``modification rules''.

Throughout this section, we fix $m$ to be any nonnegative integer. We continue to use the notion of admissible partition.

\begin{construction}[Modification rule -- Weyl group version] 
  Let $\cU$ be the set of integer sequences $(a_1, a_2,\dots)$. For $i>0$, let $s_i$ be the involution on $\cU$ that swaps $a_i$ and $a_{i+1}$. Let $s_0$ be the involution that swaps $a_1$ and $a_2$ and negates them both. We let $W$ be the group generated by the $s_i$, for $i \ge 0$.  Then $W$ is a Coxeter group of type $\rD_{\infty}$, so it is equipped with a length function $\ell \colon W \to \bZ_{\ge 0}$ with respect to the generators $s_0,s_1,\dots$, i.e., $\ell(w) = d$ if $w$ can be written as a product of $d$ of the $s_i$ (possibly with repetition), but not as a product of less than $d$ of them.

  Let $\rho = -\frac12( 2m+1, 2m+3, 2m+5,\dots)$. Define a dotted action of $W$ on $\cU$ by $w \bullet \lambda=w(\lambda+\rho)-\rho$. Given a partition $\mu$, we interpret it as an element of $\cU$ in the evident way: $(\mu_1, \mu_2, \dots)$ (padded with infinitely many 0's at the end). Then exactly one of the following two possibilities hold:
\begin{compactitem}
\item There exists a unique element $w \in W$ such that $w \bullet \mu^T$ is a partition (which we denote $\lambda^T$) such that $\lambda_1^T+\lambda_2^T \le 2m+1$ (i.e., $\lambda$ is admissible). We then put $i_{2m|1}(\mu)=\ell(w)$ and $\tau_{2m|1}(\mu)= \lambda$.
\item There exists a non-identity element $w \in W$ such that $w \bullet \mu^T=\mu^T$.  We then put $i_{2m|1}(\mu)=\infty$ and leave $\tau_{2m|1}(\mu)$ undefined. \qedhere
\end{compactitem}
\end{construction}

The above procedure may be equivalently described via a rule that involves iteratively removing border strips of certain sizes. A \defi{border strip} is a connected skew shape (two boxes are considered to be in the same connected component if they share an edge) that does not contain any $2 \times 2$ subdiagram.

\begin{construction}[Modification rule -- border strip version] 
    Let $\mu$ be a partition and write $\mu^T = (\mu_1' , \dots , \mu_n')$. We describe a modification rule to obtain the admissible partition $\lambda$ satisfying $w \bullet (\lambda^T) = \mu^T$, if such a $\lambda$ exists. We first introduce an auxiliary function $\tau'_{2m|1}$.
    \begin{enumerate}
    \item If $\mu$ is admissible, then $i_{2m|1} (\mu) \coloneq 0$ and $\tau'_{2m|1} (\mu) \coloneq \mu$.
      
        \item If $\mu$ is not admissible, then, if possible, remove a border strip $R_\mu$ of length $2 \lambda_1^T - 2m - 1$ from $\mu$ (starting in the first column). If no such border strip exists, set $i_{2m|1} (\mu) = \infty$ and leave $\tau'_{2m|1} (\mu)$ and $\tau_{2m|1}(\mu)$ undefined.  

        \item If the border strip exists, set $i_{2m|1} (\mu) \coloneq c(R_\mu) - 1 + i_{2m|1} (\mu \backslash R_\mu)$ and $\tau'_{2m|1} (\mu) \coloneq \tau'_{2m|1} (\mu \backslash R_\mu)$, where $c (R_\mu)$ denotes the number of columns in the border strip $R_\mu$.  

        \item Perform steps $(1) - (3)$ until either $i_{2m|1}(\mu) = \infty$ or one obtains an admissible partition $\lambda$. Finally, set $\tau_{2m|1}(\mu)=\tau'_{2m|1}(\mu)$ if the total number of border strips removed was even and set $\tau_{2m|1}(\mu) = \tau'_{2m|1}(\mu)^\sigma$ otherwise. \qedhere
    \end{enumerate}
\end{construction}

These two constructions agree by \cite[Proposition 4.3]{sam2013homology}. Now we can describe the specialization map on K-classes
\[
  \pi_{\Sp\bO(2m|1)} \colon \rK(\Sp\bO(\bU)) \to \rK(\Sp\bO(2m|1)).
\]

\begin{lemma} \label{lem:spo-spec}
  We have
  \[
    \pi_{\Sp\bO(2m|1)}([\bS_{[\lambda]}(\bU) ])=
    (-1)^{i_{2m|1} (\lambda)} [\bS_{[\tau_{2m|1} (\lambda)]} (\bC^{2m|1})].
  \]
\end{lemma}

\begin{proof}
  This follows from Proposition~\ref{prop:comm-square} and the fact that $\pi_{\bO(2m+1)}$ is computed the same way (see \cite[(4.4.6)]{stability-patterns} for the connection).
\end{proof}

The following property will be used later.

\begin{lemma} \label{lem:tor1}
  Let $\lambda$ be an admissible partition. There exists a unique $\mu$ such that $\tau_{2m|1}(\mu)=\lambda$ and $i_{2m|1}(\mu)=1$. Furthermore, $\mu_1^T + \mu_2^T \ge 2m+3$.
\end{lemma}

\begin{proof}
  Using the Weyl group version of the modification rule, we see that we must have $\mu = s_0 \bullet \lambda$, so that $\mu_1^T = 2m+2 - \lambda_2^T$, $\mu_2^T = 2m+2-\lambda_1^T$, and $\mu_i^T=\lambda_i^T$ for all $i \ge 3$. Since $\lambda_1^T + \lambda_2^T \le 2m+1$, we see that $\mu_1^T + \mu_2^T \ge 2m+3$.
\end{proof}

\section{Homology of Littlewood complexes for $\Sp\bO(2m | 1)$} \label{sec:homologyOfLittlewood}

\subsection{Resolving the ideals $\bS_{(2^{2m+2})} (E) \subset S^\bullet \left( \bigwedge^2 E \right)$} \label{ss:ideal}

In this section, we compute the terms of the equivariant minimal free resolution of specific ideals within the ring \( S^\bullet (\bigwedge^2  E) \). These calculations are subsumed by those in the next section, but we have done them separately both as a warmup and because the formulas simplify quite a bit in this case. Leveraging the adaptation of Bott's algorithm for Frobenius coordinates developed in \S\ref{subsec:BWBFrob}, we provide explicit descriptions of these resolutions. This approach emphasizes the natural interplay between the parametrization of weights and the geometry of the associated resolutions in a similar manner to the approach taken in \cite{sam2015littlewood}.

Consider the ideals generated by the subspace $\bS_{(2^{2m+2})} (E) \subset S^{2m+2} \left( \bigwedge^2 E \right)$ for some integer $m \geq 0$. Concretely, thinking of $S^\bullet(\bigwedge^2 E)$ as functions on $n \times n$ skew-symmetric matrices, this is the $\GL(E)$-subrepresentation generated by the determinant (not Pfaffian!) of any principal submatrix of size $2m+2$ (this is not the same thing as the ideal of minors of size $2m+2$): the determinant of the upper left $(2m+2) \times (2m+2)$ submatrix is a highest weight vector of weight $2^{2m+2}$.

\begin{example}
  In the case $m = 0$, upon identifying $S^\bullet (\bigwedge^2 E) = \bC [x_{i,j} \mid i < j ]$, the ideal $\bS_{(2,2)} (E)$ has generators of the form $x_{i,k} x_{j, \ell} + x_{j,k} x_{i,\ell}$ where $i<k$, $i \le j$, $k \le \ell$, and $j < \ell$ (if $j > k$ then this becomes $x_{i,k} x_{j, \ell} - x_{k,j} x_{i,\ell}$). The quotient by this ideal is just the even degree subalgebra of the exterior algebra $\bigwedge^\bullet E$ (i.e., the second Veronese subalgebra), which is also precisely the even spinor representation of $\so (E^* \oplus E)$ (up to a factor of $1/2$-times the trace representation).
\end{example}

To obtain the resolutions for any $m \geq 0$ we use a relative version of the resolutions induced by the resolution of the $m=0$ case and take cohomology. The statement of Corollary \ref{cor:easyFrobCase} will allow us to give a uniform combinatorial description of these resolutions.

\begin{proposition}\label{lem:S22Resolution}
  There is an equality
    $$\tor_i^{S^\bullet \left( \bigwedge^2 E \right)} \left( \frac{S^\bullet (\bigwedge^2 E)}{(\bS_{(2,2)} (E))} , \bC \right) = \bigoplus_{\substack{\underline{a} \in Q_{0}^{\rm ev} \\ \deg \underline{a} = i} } \bS_{\underline{a}} (E).$$
\end{proposition}

\begin{proof}
  By \cite[Theorem 6.3.1(c)]{weyman2003} (see also \cite[Proposition 2.4]{sam2013homology}), we have for any vector space $F$:
  \[
    \tor_i^{S^\bullet ( S^2 F )} \left( \frac{S^\bullet (S^2 F)}{(\bS_{(2,2)} F)} , \bC \right) = \bigoplus_{\substack{\underline{a} \in Q_{0}^{\rm ev} \\ \deg \underline{a} = i} } \bS_{\underline{a}} (F).
  \]
  Since this holds for any $F$, it is an identity of polynomial functors, and hence we can apply the transpose operation (see \cite[\S 7.4]{introtca}) to both sides. On Schur functors this has the effect $\bS_\lambda \mapsto \bS_{\lambda^T}$. The right hand side consists of self-conjugate partitions (i.e., $\lambda^T = \lambda$), so they are unaffected, while $(S^\bullet \circ S^2)^T = (S^\bullet \circ \bigwedge^2)$ by \cite[(7.4.8)]{introtca}. Hence we get the identity
    \[
    \tor_i^{S^\bullet ( \bigwedge^2 F )} \left( \frac{S^\bullet (\bigwedge^2 F)}{(\bS_{(2,2)} F)} , \bC \right) = \bigoplus_{\substack{\underline{a} \in Q_{0}^{\rm ev} \\ \deg \underline{a} = i} } \bS_{\underline{a}} (F)
  \]
  for all $F$. Now take $F=E$.
\end{proof}

\begin{remark} \label{rmk:SO-coinv}
  It follows from the proof of \cite[Proposition 6.3.3]{weyman2003} that there is an $S^\bullet(S^2 E)$-module $M$ with
  \[
    \tor_i^{S^\bullet ( S^2 E )} ( M , \bC ) = \bigoplus_{\substack{\underline{a} \in Q_{0}^{\rm od} \\ \deg \underline{a} = i} } \bS_{\underline{a}} (E).
  \]
  Using the same transpose trick as above, we can construct an $S^\bullet(\bigwedge^2 E)$-module $N$ such that
    \[
    \tor_i^{S^\bullet ( \bigwedge^2 E )} ( N , \bC ) = \bigoplus_{\substack{\underline{a} \in Q_{0}^{\rm od} \\ \deg \underline{a} = i} } \bS_{\underline{a}} (E). \qedhere
  \]
\end{remark}

It turns out that an identical description may be used for the terms of the resolutions of $(\bS_{(2^{2m+2})} (E)) \subset S^\bullet (\bigwedge^2 E)$, where we only need to replace $Q_0^{\rm ev}$ with $Q_{2m}^{\rm ev}$: 

\begin{proposition}\label{lem:theIdealResolution}
    There is an equality
    \[
      \tor_i^{S^\bullet \left(\bigwedge^2 E \right)} \left( \frac{S^\bullet (\bigwedge^2 E)}{(\bS_{(2^{2m+2})} (E) )} , \bC \right) = \bigoplus_{\substack{\ul{a} \in Q_{2m}^{\rm ev} \\ \deg \ul{a} = i} } \bS_{\ul{a}} (E) .
    \]
\end{proposition}

\begin{proof}
  Set $A = S^\bullet(\bigwedge^2 E)$ and $\cA \coloneq A\otimes \cO_{\Gr(n-m,E)}$. (We are continuing to use the notation $n=\dim E$.) Consider the tautological rank $n-m$ subbundle $\cR$ on $\Gr (n-m, E)$. We view $\cA$ as a flat extension of $S^\bullet (\bigwedge^2 \cR)$, and consider the global version of the complex induced by Proposition \ref{lem:S22Resolution} (tensored with $\cat{A}$), which is a resolution by vector bundles
    \[
      \cdots \to \cat{A} \otimes F_i (\cR) \to \cdots \to \cat{A} \otimes F_1 (\cR) \to \cat{A} \otimes F_0 (\cR), \qquad    F_i (\cR) \coloneq \bigoplus_{\substack{\underline{a} \in Q^{\rm ev}_0 \\ \deg \underline{a} = i}} \bS_{\ul{a}} (\cR).
    \]
    Taking the derived pushforward to $\Spec(A)$ gives a minimal complex $\bF_\bullet$ with terms
    \[
      \bF_i = \bigoplus_{j \geq 0} \rH^{j-i} ( \Gr(n-m,E), F_j (\cR)) \otimes A.
    \]
    For integers $i,d \ge 0$ with $d$ even, Corollary~\ref{cor:easyFrobCase} gives
    \[
      \rH^{md} (\Gr(n-m,E), F_i (\cR)) = \bigoplus_{\substack{\ul{a} \in Q_0^d \\ a_d \ge m \\ \deg \ul{a} = i}} \bS_{(\ul{a} - m^d \mid \ul{b} + m^d)} (E),
    \]
    and cohomology vanishes in all degrees not divisible by $m$. Since $a_1>\cdots>a_d$, the condition $a_d \ge m$ implies $\deg \ul{a} \ge \binom{d}{2} + md$. This implies that these terms contribute to $\bF_i$ where $i \ge \binom{d}{2}$, and in particular, $\bF_i=0$ if $i<0$. Also, we claim that
    \begin{align*}
      \bF_0 = A, \qquad \bF_1 = \bS_{(2^{2m+2})}(E) \otimes A.
    \end{align*}
    The contributions to $\bF_0,\bF_1$ come from $\ul{a}$ such that $\deg \ul{a} - md \le 1$, which means that $\binom{d}{2} \le 1$, i.e., $d =0$ or $d=2$ (since $d$ must be even). The case $d=0$ just means that $\ul{a}=0$, while for $d=2$ we must have $a_1=m+1$ and $a_2=m$ to get a contribution to $\bF_1$.

    In particular, $\bF_\bullet$ is the minimal free resolution of $A/(\bS_{2^{2m+2}}(E))$. Hence we get
    \begin{align*}
      \tor_i^A (A/(\bS_{(2^{2m+2})} (E) ) , \bC )
      &= \bigoplus_{j \geq 0} \rH^{j-i} ( \Gr(n-m,E), F_j (\cR))\\
      &= \bigoplus_{d\ {\rm even}} \rH^{md}(\Gr(n-m,E), F_{md+i}(\cR))\\
      &= \bigoplus_{d\ {\rm even}} \bigoplus_{\substack{\ul{a} \in Q_0^d\\ a_d \ge m\\ \deg \ul{a} = md+i}} \bS_{(\ul{a} - m^d \mid \ul{b} + m^d)} (E)\\
      &= \bigoplus_{\substack{\ul{a} \in Q_{2m}^{\rm ev} \\ \deg \ul{a} = i} } \bS_{\ul{a}} (E) . \qedhere
    \end{align*}
\end{proof}

\begin{remark}
  By transposing Schur functors, we can also use this to compute the Tor groups of the ideal generated by $\bS_{(2m+2,2m+2)}(E)$ in $S^\bullet(S^2 E)$.
\end{remark}

\begin{example}
    Lemma \ref{lem:theIdealResolution} above shows that as $m$ varies, the irreducible representations in each homological degree are parametrized by the same set of objects: all decreasing tuples of integers with even length. However, the partitions corresponding to the terms change as $m$ changes (we employ the shorthand mentioned in Remark \ref{rk:FrobAbuse} below): 
    { \tiny 
\[\begin{tikzcd}
	{\text{Frob. coords:}} & \begin{array}{c} \begin{matrix} (3,2) \\\oplus \\ (4,1) \\ \oplus \\ (5,0) \end{matrix} \end{array} & \begin{array}{c} \begin{matrix} (3,1) \\ \oplus \\ (4,0)\end{matrix} \end{array} & \begin{array}{c} \begin{matrix} (2,1) \\ \oplus \\ (3,0) \end{matrix} \end{array} & {(2,0)} & {(1,0)} & \varnothing \\
	{m=0:} & \begin{array}{c} \begin{matrix} (4,4,2,2) \\\oplus \\ (5,3,2,1,1) \\ \oplus \\ (6,2,1,1,1,1) \end{matrix} \end{array} & \begin{array}{c} \begin{matrix} (4,3,2,1) \\ \oplus \\ (5,2,1,1,1) \end{matrix} \end{array} & \begin{array}{c} \begin{matrix} (3,3,2) \\ \oplus \\ (4,2,1,1) \end{matrix} \end{array} & {(3,2,1)} & {(2,2)} & \varnothing \\
	{m=1:} & \begin{array}{c} \begin{matrix} (4,4,2,2,2^2) \\ \oplus \\ (5,3,2,2^2,1,1) \\ \oplus \\ (6,2,2^2,1,1,1,1) \end{matrix} \end{array} & \begin{array}{c} \begin{matrix}  (4,3,2,2^2,1) \\ \oplus \\ (5,2,2^2,1,1,1) \end{matrix} \end{array} & \begin{array}{c} \begin{matrix} (3,3,2,2^2) \\ \oplus \\ (4,2,2^2,1,1) \end{matrix} \end{array} & {(3,2,2^2,1)} & {(2,2,2^2)} & \varnothing
	\arrow["\cdots", from=1-1, to=1-2]
	\arrow[from=1-2, to=1-3]
	\arrow[from=1-3, to=1-4]
	\arrow[from=1-4, to=1-5]
	\arrow[from=1-5, to=1-6]
	\arrow[from=1-6, to=1-7]
	\arrow["\cdots", from=2-1, to=2-2]
	\arrow[from=2-2, to=2-3]
	\arrow[from=2-3, to=2-4]
	\arrow[from=2-4, to=2-5]
	\arrow[from=2-5, to=2-6]
	\arrow[from=2-6, to=2-7]
	\arrow["\cdots", from=3-1, to=3-2]
	\arrow[from=3-2, to=3-3]
	\arrow[from=3-3, to=3-4]
	\arrow[from=3-4, to=3-5]
	\arrow[from=3-5, to=3-6]
	\arrow[from=3-6, to=3-7]
      \end{tikzcd}\]}
  
Moving from the $m= 0$ case to the $m > 0$ case amounts to adding $2m$ boxes to each of the first $d$ columns, where $d$ is the rank of the partition. 
\end{example}

\subsection{Equivariant resolutions of the modules $M_\lambda$}\label{subsec:MlambdaRes}

In this section we define a class of $\fgl(E)$-equivariant $S^\bullet(\bigwedge^2 E)$-modules $M_\lambda$ which are analogous to those studied in \cite{sam2013homology} and which specialize to $S^\bullet(\bigwedge^2 E)/(\bS_{2^{2m+2}} E)$ in the case $\lambda=0$. We compute the terms of the minimal $\gl (E)$-equivariant $S^\bullet (\bigwedge^2 E)$-free resolutions of these modules and use this to compute the homology of the $\Sp\bO(2m|1)$-Littlewood complexes.

\begin{definition}\label{def:MlambdaMods}
  Let $\lambda$ be a partition with $\ell(\lambda) \le m$ and $E$ a vector space with $\dim E = n$.

\begin{enumerate}[(1)]
\item   If $n > m$, define $X \coloneq \Gr (n - m , E)$ with tautological rank $n  - m$ subbundle $\cR$ and tautological rank $m$ quotient bundle $\cQ$. Set
    \begin{align*}
      \cM_\lambda &\coloneq \coker \left( S^\bullet (\bigwedge^2 E) \otimes  \bS_{(2,2)} (\cR) \otimes_{\cO_X} \bS_\lambda (\cat{Q})  \to S^\bullet (\bigwedge^2 E) \otimes  \bS_\lambda (\cat{Q}) \right)\\
        \cM_{\lambda^\sigma} &\coloneq \coker \left( S^\bullet (\bigwedge^2 E) \otimes  \bS_{(2,1)} (\cR) \otimes_{\cO_X} \bS_\lambda (\cat{Q})  \to S^\bullet (\bigwedge^2 E) \otimes \cR \otimes \bS_\lambda (\cQ) \right),
    \end{align*}
    where the map is induced by tensoring the global version of complexes of Lemma~\ref{lem:S22Resolution} and Remark~\ref{rmk:SO-coinv} with $\bS_\lambda (\cQ)$: 
    \[
      S^\bullet (\bigwedge^2 \cR) \otimes  \bS_{(2,2)} (\cR)  \to S^\bullet (\bigwedge^2 \cR), \qquad
      S^\bullet (\bigwedge^2 \cR) \otimes  \bS_{(2,1)} (\cR)  \to S^\bullet (\bigwedge^2 \cR) \otimes \cR,
    \]
    then tensored along the flat map $S^\bullet (\bigwedge^2 \cR) \to S^\bullet (\bigwedge^2 E) \otimes  \cO_X$. Finally we define $S^\bullet (\bigwedge^2 E)$-modules $M_\lambda$ and $M_{\lambda^\sigma}$ via
    \[
      M_\lambda \coloneq \rH^0 (X , \cat{M}_\lambda), \qquad       M_{\lambda^\sigma} \coloneq \rH^0 (X , \cat{M}_{\lambda^\sigma}).
    \]

  \item  If $n \le m$, we will define $M_\lambda$ to be the free $S^\bullet(\bigwedge^2 E)$-module $\bS_\lambda E \otimes S^\bullet(\bigwedge^2 E)$ and $M_{\lambda^\sigma} = 0$.    \qedhere
  \end{enumerate}
\end{definition}

For this next portion, we assume $n > m$.

To compute cohomology for sheaves on $X$, we will use Borel--Weil--Bott (Theorem~\ref{thm:bott}). We recall that for sequences $\alpha$ of length $n$ and permutations $w$, we use a dotted action $w \bullet \alpha = w(\alpha + \rho)-\rho$ where $\rho = (n-1, \dots, 1,0)$. If $\ell(\lambda)\le m$ and $\ell(\mu) \le n-m$, we define $(\lambda \mid \mu) = (\lambda_1,\dots,\lambda_m, \mu_1,\dots, \mu_{n - m})$ where we pad $\lambda$ with 0's if necessary.

We recall the following sets originally defined in \cite[Section 4]{sam2013homology}.

\begin{definition}
    A pair of partitions $(\lambda , \mu)$ is \defi{singular} if there exists a non-identity permutation $w$ such that $w \bullet (\lambda \mid \mu) = (\lambda  \mid \mu)$. A pair $(\lambda , \mu)$ is \defi{regular} otherwise. Given a partition $\lambda$ with $\ell(\lambda) \le m$, define the two sets
    \begin{align*}
      \overline{S_1} (\lambda) &\coloneq \{ \mu \in Q_0 \mid (\lambda \mid \mu) \ \text{is regular} \},\\
    \overline{S_2} (\lambda) &\coloneq \{ \alpha \mid \ol{\tau_{2m|1} (\alpha)} = \lambda \}. \qedhere
    \end{align*}
  \end{definition}
  
\begin{proposition}[{\cite[Proposition 4.14]{sam2013homology}}]\label{prop:SsetBijection}
    There is a unique bijection $\phi \colon \overline{S_1} (\lambda) \to \overline{S_2} (\lambda)$ satisfying $\phi (\mu) = \alpha$ if and only if there exists a permutation $w$ satisfying $w \bullet (\lambda \mid \mu) = \alpha$. Moreover, we have that $\ell (w) + i_{2m|1} (\phi (\mu)) = \frac12 (|\mu| - \rank (\mu))$ and
    \[
      \tau_{2m|1} (\phi (\mu)) = \begin{cases}
        \lambda & \text{if} \ \rank (\mu) \ \text{is even}, \\
        \lambda^\sigma & \text{if} \ \rank (\mu) \ \text{is odd.}
      \end{cases}
    \]
\end{proposition}

\begin{theorem}\label{thm:MlambdaTor}
Let $\lambda$ be an admissible partition. We have $\gl(E)$-equivariant isomorphisms:
\begin{align*}
  \tor_i^{S^\bullet \left( \bigwedge^2 E \right)} (M_\lambda , \bC) &= \bigoplus_{\substack{\tau_{2m|1} (\alpha) = \lambda, \\  i_{2m|1} (\alpha) = i}} \bS_{\alpha} (E).
\end{align*}
\end{theorem}

\begin{proof}
  First assume $\ell(\lambda) \le m$. Set $A = S^\bullet(\bigwedge^2 E)$.
  
  Consider the resolutions by vector bundles induced by Proposition \ref{lem:S22Resolution}:
    \[
      \cdots \to A \otimes \bS_\lambda (\cat{Q}) \otimes F_i (\cR) \to \cdots \to A \otimes \bS_\lambda (\cQ) \otimes F_1 (\cR) \to A \otimes \bS_\lambda (\cQ) \otimes F_0 (\cR),
    \]
    where 
    \[
      F_i(\cR) \coloneq \bigoplus_{\substack{\ul{a} \in Q^{{\rm ev}}_0 \\ \deg(\ul{a}) = i}} \bS_{\ul{a}} (\cR).
    \]
    Taking the derived pushforward to $\Spec(A)$ gives a minimal complex $\bF_\bullet$ with terms
    \begin{align*}
      \bF_i &= A \otimes \bigoplus_{j \geq 0} \rH^{j-i} (X, \bS_\lambda (\cQ) \otimes F_j(\cR))\\
            &= A \otimes \bigoplus_{\substack{\mu \in Q_0^{\rm ev} \\ i_{2m|1} (\phi (\mu)) = i}} \bS_{\phi (\mu)} (E),
    \end{align*}
    where the second equality follows from Proposition~\ref{prop:SsetBijection}. In particular, $\bF_i = 0$ for $i<0$, so we conclude that $\bF_i$ is a minimal free resolution of $M_\lambda$.

    If $\ell(\lambda) > m$, then we instead use $Q_0^{\rm od}$ in the definition of $F_i(\cR)$ and we use Remark~\ref{rmk:SO-coinv} instead of Proposition~\ref{lem:S22Resolution}.
\end{proof}

For the remainder of the section we allow $n$ and $m$ to be arbitrary.

\begin{corollary} \label{cor:Mlambda-reldeg}
  For any admissible partition $\lambda$, we have that $\Tor_1^{S^\bullet(\bigwedge^2 E)}(M_\lambda, \bC) = \bS_\mu(E)$ where $\mu$ is given by the formula $\mu_1^T = 2m+2-\lambda_2^T$, $\mu_2^T=2m+2-\lambda_1^T$ and $\mu_i^T=\lambda_i^T$ for $i \ge 3$. In particular, this Tor group is concentrated in a single degree or is $0$.
\end{corollary}

\begin{proof}
  This is immediate from the previous result and Lemma~\ref{lem:tor1}.
\end{proof}

\begin{lemma}
  Given an admissible partition $\lambda$, define
  \[
    S^\bullet (E \otimes U)_\lambda \coloneq \hom_{\Sp\bO (U)} (\bS_{[\lambda]} (U) , S^\bullet (E \otimes U)).
  \]
  Then there is an isomorphism of $\gl (E)$-representations
    $$M_\lambda \cong S^\bullet (E \otimes U)_\lambda.$$
\end{lemma}

\begin{proof}
    Define the $S^\bullet (\bigwedge^2 E)$-module $M \coloneq \bigoplus_{\lambda} M_\lambda \otimes  \bS_{[\lambda]} (U)$ where the sum is over all admissible partitions $\lambda$. As in \cite[Lemma 4.21]{sam2013homology}, it suffices to prove that, for any partition $\theta$, the $\bS_\theta (E)$-multiplicity spaces of $M$ and $S^\bullet (E \otimes U)$ are isomorphic as $\Sp\bO (U)$-representations, which can be accomplished by looking at the respective characters; the computation is essentially identical, but for sake of completeness we reproduce it here. Applying the specialization homomorphism (Lemma~\ref{lem:spo-spec}) to the $\Sp\bO (2m | 1)$-branching formula \eqref{eqn:infinite-2} yields the character identity
    \[
      [\bS_\theta (U)] = \bigoplus_{\mu, \nu} (-1)^{i_{2m|1} (\nu)} c_{\nu , (2\mu)^T}^\theta [\bS_{[\tau_{2m|1} (\nu)]} (U)].
    \]
    On the other hand, there is an equality
    \begingroup\allowdisplaybreaks
    \begin{align*}
        [M_\lambda] &= [S^\bullet (\bigwedge^2 E)] \cdot \sum_i (-1)^i [\tor_i^{S^\bullet (\bigwedge^2 E)} (M_\lambda , \bC) ] \\
        &= \sum_{\mu} [\bS_{(2\mu)^T} (E)] \sum_{\substack{\nu \\ \tau_{2m|1} (\nu) = \lambda}} (-1)^{i_{2m|1} (\nu)} [\bS_{\nu} (E)] \\
        &= \sum_{\substack{\nu , \mu ,\theta \\ \tau_{2m|1} (\nu) = \lambda}} (-1)^{i_{2m|1} (\nu)} c^\theta_{\nu , (2 \mu)^T} [\bS_\theta (E)].
    \end{align*}
    \endgroup
    Thus restricting to the $\bS_\theta (E)$-multiplicity spaces yields the same character, whence the desired isomorphism follows.
\end{proof}

\begin{lemma}\label{lem:SwedgeModuleIso}
    There is an $S^\bullet (\bigwedge^2 E)$-module isomorphism $M_\lambda \to S^\bullet (E \otimes U)_\lambda$. 
\end{lemma}

\begin{proof}
    The proof is again essentially identical to that given in \cite[Proposition 4.22]{sam2013homology}.
\end{proof}

To conclude this section, we prove the analogue of the results of \cite{sam2013homology} for $\Sp\bO (2m|1)$. The proof is again a balancing argument that leverages that there are multiple different ways to compute Tor, combined with our identification of the modules $M_\lambda$ as the $\Sp\bO (U)$-multiplicity modules of $S^\bullet (E \otimes U)$.

\begin{theorem}
  Let $\lambda$ be a partition and $L^\lambda_\bullet (U)$ the orthosymplectic Littlewood complex. Then
    $$\rH_i (L_\bullet^\lambda (U)) = \begin{cases}
        \bS_{[\tau_{2m|1} (\lambda)]} (U) & \text{if} \ i = i_{2m|1} (\lambda), \\
        0 & \text{otherwise.}
    \end{cases}$$
\end{theorem}

\begin{proof}
    Recall that by definition the orthosymplectic Littlewood complexes are defined via the equality
    \[
      \bigwedge^\bullet (\bigwedge^2 E) \otimes_{S^\bullet (\bigwedge^2 E)} S^\bullet (E \otimes U) =
      \bigoplus_{\lambda} \bS_\lambda (E) \otimes_{S^\bullet (\bigwedge^2 E)} L^\lambda_\bullet (U).
    \]
    On the other hand, Lemma~\ref{lem:SwedgeModuleIso} implies that there is an $S^\bullet (\bigwedge^2 E)$-module isomorphism $S^\bullet (E \otimes U) \cong \bigoplus_{\mu} M_\mu \otimes \bS_{[\mu]}(U)$. Taking $i$th homology thus yields the isomorphism
    \[
      \bigoplus_{\mu} \tor_i (M_\mu , \bC) \otimes \bS_{[\mu]} (U) \cong \bigoplus_{\lambda} \bS_\lambda (E) \otimes \rH_i (L^\lambda_\bullet (U)).
    \]
    Employing Theorem \ref{thm:MlambdaTor} and comparing $\bS_\lambda (E)$-multiplicity spaces in the above equality, the theorem follows.
  \end{proof}

This allows us to improve the statement in Proposition~\ref{prop:superRegSeq}:
  
  \begin{corollary} \label{cor:adm-acyclic}
    If $\lambda$ is admissible, then $L^\lambda_\bullet(U)$ is acyclic.
  \end{corollary}

\section{The Howe dual pair $(\Sp\bO (2m|1) , \so (2n))$} \label{sec:howe}

In this section, we establish that the orthosymplectic Lie group $ \Sp\bO(2m|1) $ and the orthogonal Lie algebra $ \so (2n) $ form a Howe dual pair with respect to the symmetric superalgebra \( S^\bullet(\bC^{2m|1} \otimes \bC^{n}) \). The key takeaway here is that this will show that the $\GL(n)$-action defined on the $\Sp\bO(2m|1)$-multiplicity spaces $M_\lambda$ considered in the previous section extends to an action of $\fso(2n)$.

This result provides a key structural insight that will be used to deduce an essential character equality. The proof presented here is a natural extension of the classical Howe duality framework, adapted to the superalgebraic setting of \( \Sp\bO (2m|1) \). See \cite[\S 8.2]{goodman2004multiplicity} for an exposition of the following for the analogous case of $\Sp(2m)$.

\subsection{Setup}
As before, let $U$ be a $2m|1$-dimensional vector superspace equipped with a (super) symplectic form $\omega$ and let $E$ be a $n$-dimensional vector space. Consider the affine space
\[
  X = \hom(E,U)
\]
under the action of $\fgl(E) \times \Sp\bO(U)$. Let $\phi$ denote a general element of $X$. Fix a basis $e_1, \dots, e_n$ for $E$ and fix a homogeneous basis $v_{-m},\dots,v_m$ for $U$ so that $v_0$ is odd and so that the pairing is defined by $\omega(v_{-i}, v_j) = -\omega(v_j,v_{-i})=\delta_{i,j}$ for $i,j>0$ and $\omega(v_0,v_0)=1$. Let $\phi_{j,i}$ be the coefficient of $v_j$ in $\phi(e_i)$.
Let $r_{i,j}$ be the quadratic polynomial $\omega(\phi(e_i), \phi(e_j))$. Explicitly, this can be written as
\[
  r_{i,j} = \phi_{0,i} \phi_{0,j} + \sum_{a=1}^m (\phi_{-a,i} \phi_{a,j} - \phi_{-a,j}\phi_{a,i}).
\]
Let $\Delta_{i,j}$ be the second-order differential operator corresponding to $r_{i,j}$:
\[
  \Delta_{i,j} = \partial_{0,i} \partial_{0,j} + \sum_{a=1}^m (\partial_{-a,i} \partial_{a,j} - \partial_{-a,j}\partial_{a,i}).
\]

Note that $\phi_{0,i}$ are odd variables so they skew-commute. That also means the Weyl algebra relation becomes $\partial_{0,i} \phi_{0,j} + \phi_{0,j} \partial_{0,i} = \delta_{i,j}$ (instead of subtracting). Hence $r_{i,j}=-r_{j,i}$ and $\Delta_{i,j}=-\Delta_{j,i}$. Also, define
\[
E_{i,j} = \sum_{a=-m}^m \phi_{a,i} \partial_{a,j} + \frac{2m-1}{2} \delta_{i,j}.
\]
We remark that $2m-1$ is the supertrace of the identity operator on $U$, which is why it shows up above. Given homogeneous operators $f,g$, we use the bracket notation $[f,g]$ to denote the supercommutator, i.e.,
\[
  [f,g] = fg - (-1)^{|f||g|} gf.
\]
We will need the following supercommutator equalities:

\begin{lemma}\label{lem:commutatorIdentities}
    With operators as defined above, there are supercommutator equalities:
    \begin{subeqns}
      \begin{align}
      [r_{i,j} , \Delta_{k,\ell}] &= - \delta_{i,k} E_{j,\ell} + \delta_{i,\ell} E_{j,k} + \delta_{k,j} E_{i , \ell} - \delta_{j,\ell} E_{i,k}, \label{eqn:comm1} \\
    [E_{i,j} , r_{k,\ell}] &= \delta_{k,j} r_{i,\ell} - \delta_{\ell , j} r_{i,k}, \label{eqn:comm2} \\
    [E_{i,j} , \Delta_{k,\ell}] &= - \delta_{i,k} \Delta_{j ,\ell} + \delta_{i,\ell} \Delta_{j,k}, \label{eqn:comm3} \\
    [E_{i,j} , E_{k,\ell}] &= \delta_{j,k} E_{i,\ell} - \delta_{i,\ell} E_{k,j} \label{eqn:comm4}.
    \end{align}
  \end{subeqns}
\end{lemma}

The proofs of these identities are moved to \S\ref{ss:calc} to avoid disrupting the exposition.

Let $\fn_-$ be the span of the $r_{i,j}$, let $\fn_+$ be the span of the $\Delta_{i,j}$, and let $\fk$ be the span of the $E_{i,j}$ in the Weyl algebra of $X$. Then $\fk \cong \fgl(n)$ as Lie algebras. These operators commute with the $\Sp\bO(U)$ action. From the previous lemma, they span a Lie subalgebra isomorphic to $\fso(E \oplus E^*)$, which has a parabolic decomposition
\begin{align} \label{eqn:parabolic}
\fso(E \oplus E^*) = \fn_- \oplus \fk \oplus \fn_+.
\end{align}
Our choice of Cartan subalgebra is the span of $E_{1,1},\dots,E_{n,n}$. Note that the choice of $\fn_\pm$ is the opposite of the standard convention: $r_{i,j}$ spans the $\eps_i + \eps_j$ root space where $\eps_i(E_{j,j})=\delta_{i,j}$. But it should now correspond to a negative root, and so the correction is to reverse and negate the weights which instead corresponds to the ordered basis $-E_{n,n},\dots,-E_{1,1}$ for the Cartan subalgebra, so that the weights $\mu$ are given by the sequence $(-\mu(E_{n,n}),\dots,-\mu(E_{1,1}))$.

Thus $\fso(E \oplus E^*)$ acts on each multiplicity space of $\Sp\bO(U)$ on $\bC[X]$, and this action is irreducible. This is proven in \cite[\S 8.1]{goodman2004multiplicity} in the classical case. For our setting, we need to adapt \cite[Proposition 5.12]{cheng2012dualities}, which does not include $G = \Sp\bO(U)$. However, its finite-dimensional representations are semisimple \cite[Corollary 2.33]{cheng2012dualities}, and that is all that is used in the proof.

As before, if $\lambda$ is an admissible partition (i.e., $\lambda_1^T + \lambda_2^T \le 2m+1$), we write $\bS_{[\lambda]}(U)$ for the irreducible representation of $\Sp\bO(U)$ with highest weight $\lambda$. We note that if $\bS_{[\lambda]}(U)$ is a summand of the Schur functor $\bS_\mu(U)$, then $\lambda \subseteq \mu$ and that it appears with multiplicity 1 if $\lambda=\mu$ (by \eqref{eqn:infinite-2}). Suppose that $\bS_{[\lambda]}(U)$ is a summand of $\bC[X]$. From the Cauchy identity, we have a decomposition
\[
  \bC[X] = S^\bullet(U \otimes E) \cong \bigoplus_\mu \bS_{\mu}(U) \otimes \bS_{\mu}(E)
\]
(using that $U \cong U^*$ via its super symplectic form). From this, we see that $\lambda_i=0$ if $i>n$. Define
the length $n$ sequence
\[
  \zeta' = -\frac12 (2m-1, 2m-1, \dots, 2m-1) = (\frac12 - m, \dots, \frac12 - m).
\]
From our discussion above, the highest weight of the action of $\fgl(E)$ on the $\bS_{[\lambda]}(U)$-multiplicity space of $\bC[X]$ is 
\[
\zeta'-\lambda^{\rm op} = (-\lambda_n- \frac{2m-1}{2}, -\lambda_{n-1}- \frac{2m-1}{2}, \dots, -\lambda_1 - \frac{2m-1}{2})
\]
(if $n>m$, then we use the convention that $\lambda_i=0$ for $i>m$).

Summarizing the discussion so far gives the following result.

\begin{proposition} \label{prop:spo-howe}
    We have an $\Sp\bO(U) \times \fso(E \oplus E^*)$-equivariant decomposition
\begin{align*}
  \bC[X] = S^\bullet(U \otimes E) = \bigoplus_\lambda \bS_{[\lambda]}(U) \otimes L_{\fso(E \oplus E^*)}(\zeta' - \lambda^{\rm op})
\end{align*}
where $L_\fg(\mu)$ denotes the irreducible $\fg$-representation with highest weight $\mu$, and where the sum is over all partitions $\lambda$ such that $\lambda_1^T + \lambda_2^T \le 2m+1$ and $\lambda_1^T \le n$.
\end{proposition}

\begin{remark} \label{rmk:iso-mult}
  In order to make the connection to the previous section more clear, note that there is an isomorphism $M_\lambda \cong L_{\fso(E \oplus E^*)}(\zeta' - \lambda^{\rm op})$. Thus these Lie algebra representations are precisely the $\Sp \bO (U)$-multiplicity spaces of $S^\bullet (E \otimes U)$ (as in Lemma \ref{lem:SwedgeModuleIso}).
\end{remark}

\subsection{Calculations} \label{ss:calc}

We begin by stating the following standard formula for supercommutators, which will be used repeatedly.

\begin{lemma} \label{lem:quad}
Let $R$ be an associative superalgebra. Given homogeneous elements $a,b,c,d \in R$, we have
\[
  [ab,cd] = a [b,c] d + (-1)^{|b| \cdot |c|} [a,c] b d + (-1)^{(|a|+|b|)|c|} c a [b,d] + (-1)^{|a| \cdot |c| + |b| (|d| + |c|)} c [a,d] b.
\]
\end{lemma}

    \textbf{Proof of \eqref{eqn:comm1}:} From Lemma~\ref{lem:quad}, we get:
    \begin{align*}
        [\phi_{0,i} \phi_{0,j} , \partial_{0,k} \partial_{0,\ell}] =& \delta_{j,k} \phi_{0,i} \partial_{0,\ell} - \delta_{i,k} \phi_{0,j} \partial_{0,\ell} + \delta_{j,\ell} \partial_{0,k} \phi_{0,i} - \delta_{i,\ell} \partial_{0,k} \phi_{0,j} \\
      =& \delta_{j,k} \phi_{0,i} \partial_{0,\ell} - \delta_{i,k} \phi_{0,j} \partial_{0,\ell} - \delta_{j , \ell} \phi_{0,i} \partial_{0,k} + \delta_{i,\ell} \phi_{0,j} \partial_{0,k} + \delta_{j,\ell} \delta_{i,k} - \delta_{i,\ell} \delta_{j,k},
    \end{align*}
    and for $a \ne 0$:
    \begin{align*}
      [\phi_{-a,i} \phi_{a,j} , \partial_{-a,k} \partial_{a,\ell}] =& -\delta_{i,k} \phi_{a,j} \partial_{a,\ell} - \delta_{j,\ell} \partial_{-a,k} \phi_{-a,i} \\
        =& -\delta_{i,k} \phi_{a,j} \partial_{a,\ell} - \delta_{j,\ell} \phi_{-a,i} \partial_{-a,k} - \delta_{j,\ell} \delta_{i,k}.
    \end{align*}
    Using this, we compute:
    \begingroup\allowdisplaybreaks
    \begin{align*}
        [r_{i,j} , \Delta_{k,\ell}] =& \delta_{j,k} \phi_{0,i} \partial_{0,\ell} - \delta_{i,k} \phi_{0,j} \partial_{0,\ell} - \delta_{j , \ell} \phi_{0,i} \partial_{0,k} + \delta_{i,\ell} \phi_{0,j} \partial_{0,k} + \delta_{j,\ell} \delta_{i,k} - \delta_{i,\ell} \delta_{j,k} \\
        &- \sum_{a=1}^m \left( \delta_{i,k} \phi_{a,j} \partial_{a,\ell} + \delta_{j,\ell} \phi_{-a,i} \partial_{-a,k} + \delta_{j,\ell} \delta_{i,k} \right) \\
        &+ \sum_{a=1}^m \left( \delta_{j,k} \phi_{a,i} \partial_{a,\ell} + \delta_{i,\ell} \phi_{-a,j} \partial_{-a,k} + \delta_{j,k} \delta_{i,\ell} \right) \\
        &+ \sum_{a=1}^m \left( \delta_{i,\ell} \phi_{a,j} \partial_{a,k} + \delta_{j,k} \phi_{-a,i} \partial_{-a,\ell} + \delta_{i,\ell} \delta_{j,k} \right) \\
        &- \sum_{a=1}^m \left( \delta_{j,\ell} \phi_{a,i} \partial_{a,k} + \delta_{i,k} \phi_{-a,j} \partial_{-a,\ell} + \delta_{j,\ell} \delta_{i,k} \right) \\
        =& - \delta_{i,k} \sum_{a = -m}^m \phi_{a,j} \partial_{a,\ell} + \delta_{i,\ell} \sum_{a=-m}^m \phi_{-a,j} \partial_{-a,k} + \delta_{j,k} \sum_{a=-m}^m \phi_{a,i} \partial_{a,\ell} \\
        &-\delta_{j,\ell} \sum_{a=-m}^m \phi_{-a,i} \partial_{-a,k} - (2m-1) \delta_{i,k} \delta_{j,\ell} + (2m-1) \delta_{j,k} \delta_{i,\ell} \\
        =& - \delta_{i,k} E_{j,\ell} + \delta_{i,\ell} E_{j,k} + \delta_{j,k} E_{i,\ell} - \delta_{j,\ell} E_{i,k}.
    \end{align*}
    \endgroup
    \textbf{Proof of \eqref{eqn:comm2}:} We first note the following commutator identities (below, $a \ne 0$):
    \begingroup\allowdisplaybreaks
    \begin{align*}
        [\phi_{0,i} \partial_{0,j} , \phi_{0,k} \phi_{0,\ell}] &= \delta_{j,k} \phi_{0,i} \phi_{0,\ell} - \delta_{j,\ell} \phi_{0,i} \phi_{0,k}, \\
        [\phi_{a,i} \partial_{a,j} , \phi_{-a,k} \phi_{a,\ell}] &= \delta_{j,\ell} \phi_{a,i} \phi_{-a,k} , \quad \text{and} \\
        [\phi_{-a,i} \partial_{-a,j} , \phi_{-a,k} \phi_{a,\ell} ] &= \delta_{j,k} \phi_{-a,i} \phi_{a,\ell} .
    \end{align*}
    \endgroup
    \begingroup\allowdisplaybreaks
    \begin{align*}
        [E_{i,j} , r_{k , \ell} ] &= \delta_{j,k} \phi_{0,i} \phi_{0,\ell} - \delta_{j,\ell} \phi_{0,i} \phi_{0,k} 
        + \sum_{a=1}^m \left( \delta_{j,\ell} \phi_{a,i} \phi_{-a,k} - \delta_{j,k} \phi_{a,i} \phi_{-a,\ell} \right) \\
        &+ \sum_{a=1}^m \left( \delta_{j,k} \phi_{-a,i} \phi_{a,\ell} - \delta_{j ,\ell} \phi_{-a,i} \phi_{a,k} \right) \\
        &= \delta_{j,k} r_{i,\ell} - \delta_{j,\ell} r_{i,k}.
    \end{align*}
    \endgroup
    \textbf{Proof of \eqref{eqn:comm3}:} We note the following analogous commutator identities:
    \begingroup\allowdisplaybreaks
    \begin{align*}
        [\phi_{0,i} \partial_{0,j} , \partial_{0,k} \partial_{0,\ell}] =& -\delta_{i,k} \partial_{0,j} \partial_{0,\ell} + \delta_{i,\ell} \partial_{0,j} \partial_{0,k}, \\
        [\phi_{a,i} \partial_{a,j} , \partial_{-a,k} \partial_{a,\ell} ] =& - \delta_{i,\ell} \partial_{-a,k} \partial_{a,j} , \quad \text{and} \\
        [\phi_{-a,i} \partial_{-a,j} , \partial_{-a,k} \partial_{a,\ell}] = & - \delta_{i,k} \partial_{-a,j} \partial_{a,\ell}.
    \end{align*}
    \endgroup
    A similar computation to the previous case yields:
    \begingroup\allowdisplaybreaks
    \begin{align*}
        [E_{i,j} , \Delta_{k,\ell}] =& -\delta_{i,k} \partial_{0,j} \partial_{0,\ell} + \delta_{i , \ell} \partial_{0,j} \partial_{0,k} 
        + \sum_{a=1}^m \left( -\delta_{i,\ell}  \partial_{-a,k} \partial_{a,j} + \delta_{i,k} \partial_{-a,\ell} \partial_{a,j} \right) \\ 
        &\qquad + \sum_{a=1}^m \left( -\delta_{i ,k} \partial_{-a,j} \partial_{a,\ell} + \delta_{i,\ell} \partial_{-a,j} \partial_{a,k} \right) \\
        =& - \delta_{i,k} \Delta_{j,\ell} + \delta_{i,\ell} \Delta_{j,k}.  
    \end{align*}
    \endgroup
    \textbf{Proof of \eqref{eqn:comm4}:} For any $a$, we have the following identity:
    \[
      [\phi_{a,i} \partial_{a,j} , \phi_{a,k} \partial_{a,\ell} ] = \delta_{j,k} \phi_{a,i} \partial_{a,\ell} - \delta_{i,\ell} \phi_{a,k} \partial_{a,j}
    \]
    which implies
    \begin{align*}
      [E_{i,j} , E_{k,\ell}] &= \sum_{a=-m}^m \left( \delta_{j,k} \phi_{a,i} \partial_{a,\ell} - \delta_{i,\ell} \phi_{a,k} \partial_{a,j} \right) \\
        &= \delta_{j,k} E_{i,\ell} - \delta_{i,\ell} E_{k,j}. 
    \end{align*}

\section{General odd-rank quadric hypersurfaces}\label{sec:generalOddRank}

Let $V$ be a $(2m+1)$-dimensional vector space equipped with a nondegenerate symmetric bilinear form $\beta$ and let $E$ be an $n$-dimensional vector space.

\subsection{The Lie algebra $\g_{V,E}$ as a parabolic subalgebra}

In this section, we reinterpret the Lie algebras \(\fg_{V,E}\) from the introduction as the nilpotent radical of a parabolic subalgebra within a larger orthogonal Lie algebra \(\fso(\bE)\). This perspective allows us to construct a parabolic Verma module \(M(\zeta)\), from which we define a quotient representation \(Z_{V,E}\). The character of \(Z_{V,E}\) will be a central focus of subsequent computations, as the explicit form of its equivariant Hilbert function is essential for the content of \cite{sam2024total}.

Define
\[
  \bE = E^* \oplus V \oplus E.
\]
We can extend $\beta$ to a nondegenerate symmetric bilinear form on $\bE$ via the following formula:
    \[
      \beta( (f, v, e), (f', v', e') ) = \beta(v,v') + f(e') + f'(e).
    \]
    We define a $\bZ$-grading on $\bE$ by
    \[
      \deg(E^*)=-1, \qquad \deg(V)=0, \qquad \deg(E)=1.
    \]
    This induces a $\bZ$-grading on $\fso(\bE) \cong \bigwedge^2 \bE$ (supported on $-2$ to $2$) with decomposition
    \[
      (\bigwedge^2 E^*) \oplus (E^* \otimes V) \oplus \begin{pmatrix} \fgl(E) \\ \times \\ \fso(V) \end{pmatrix} \oplus (E \otimes V) \oplus (\bigwedge^2 E).
    \]
    Note that $\fg_{V,E} = \fso(\bE)_{>0}$ is the positive portion of this algebra. The subalgebras $\fso(\bE)_{\ge 0}$ and $\fso(\bE)_{\le 0}$ are both examples of parabolic subalgebras and $\fso(\bE)_{>0}$ and $\fso(\bE)_{<0}$ are their respective nilpotent radicals.

 More generally, let $\fp$ be a parabolic subalgebra of a semisimple Lie algebra $\fg$, let $\fn = [\fp,\fp]$ be its nilpotent radical, and let $\fg_0$ be the Levi subalgebra of $\fp$, which is isomorphic to $\fp/\fn$. Given a finite-dimensional representation $F$ of $\fp$, we define the induction
    \[
      {\rm Ind}_\fp^\fg F = \rU(\fg) \otimes_{\rU(\fp)} F.
    \]
    By PBW, this is isomorphic (as a $\fg_0$-representation) to
    \[
      \rU(\fn) \otimes F
    \]
    If $F$ is a representation of $\fg_0$, we will treat it as a $\fp$-representation by pulling back along the surjection $\fp \to \fp/\fn \cong \fg_0$.

    For our purposes, we will take $\fg=\fso(\bE)$ and $\fp = \fso(\bE)_{\le 0}$ (then $\fg_0 = \fso(V) \times \fgl(E)$). We write weights of $\fso(\bE)$ as sequences of complex numbers $(\lambda_1,\dots,\lambda_{m+n})$. When restricted to $\fgl(E) \times \fso(V)$, we get two weights $(\lambda_1,\dots, \lambda_n)$ and $(\lambda_{n+1},\dots,\lambda_{m+n})$. If these are both dominant (i.e., $\lambda_1 \ge \cdots \ge \lambda_n$ and $\lambda_i - \lambda_{i+1} \in \bZ_{\ge 0}$ for $i<n$; $\lambda_{n+1} \ge \cdots \ge \lambda_{m+n} \ge 0$ and $\lambda_i \in \bZ_{\ge 0}$ for $i>n$), then the irreducible $\fgl(E) \times \fso(V)$-representation with highest weight $\lambda$ is finite-dimensional. Its induction is the parabolic Verma module with respect to $\fp$, and we will denote it $M(\lambda)$.

\begin{remark} \label{rmk:hw}
  Since we have chosen $\fso(\bE)_{\le 0}$ as our parabolic, there will be some reversal of conventions when speaking about Schur functors for $\fgl(E)$ (since standard conventions are to use the upper-triangular matrices in $\fgl(E)$ as the Borel subalgebra, but $\fso(\bE)_{\le 0} \cap \fgl(E)$ is the space of lower-triangular matrices). In particular, $M(\lambda) = {\rm Ind}_\fp^{\fso(\bE)} F$, where $F$ is the tensor product $\bC_{-\lambda_1} \otimes \bS_\alpha(E) \otimes F'$ where $\bC_a$ is $a$ times the trace character, $\alpha = (\lambda_1-\lambda_n, \lambda_1-\lambda_{n-1}, \dots, \lambda_1-\lambda_2, 0)$, and $F'$ is the irreducible representation of $\fso(V)$ with highest weight $(\lambda_{n+1},\dots,\lambda_{m+n})$.
\end{remark}

Define the weight
\begin{align} \label{eqn:weight}
\zeta =  ( \underbrace{\frac12 - m, \frac12 - m, \dots, \frac12 - m}_n, \underbrace{0, 0, \dots, 0}_m ).
\end{align}
We highlight that the weight $\zeta'$ defined in the previous section is the first $n$ entries of $\zeta$.

As in Remark~\ref{rmk:hw}, $M(\zeta) = {\rm Ind}_\fp^{\fso(\bE)} \bC_{m-\frac12}$. In particular, $M(\zeta) \cong \rU(\fg_{V,E})$ as a representation of $\fgl(E) \times \fso(V)$ (up to tensoring with the appropriate multiple of the trace function).

As explained in the introduction, there is a natural inclusion $S^2 E \subset \rU(\fg_{V,E})$. The following proposition is our first goal.

\begin{proposition} \label{prop:main}
  The left $\rU(\fg_{V,E})$-module generated by $S^2 E$ is an $\fso(\bE)$-subrepresentation of $M(\zeta)$.
\end{proposition}

We let $Z_{V,E}$ denote the quotient of $M(\zeta)$ by the left $\rU(\fg_{V,E})$-module generated by $S^2 E$.

\begin{example} \label{ex:interpolate}
There are two extremal examples for the quotient $Z_{V,E}$:
\begin{itemize}
    \item If $\dim V = 1$, then $\fso(\bE) \cong \fso(2n+1)$ and $Z_{V,E}$ is the spinor representation.
    \item If $\dim E = 1$, then $\fso(\bE) \cong \fso(2m+3)$ and $Z_{V,E}$ is the homogeneous coordinate ring of a nonsingular quadric hypersurface in $m+1$ variables.
\end{itemize}
In particular, the quotient $Z_{V,E}$ interpolates between these two representations as the dimensions of $E$ and $V$ vary. 
\end{example}

The rest of this section is devoted to the proof of Proposition \ref{prop:main}, but we will first need to recall some preliminaries. While the notions below are standard, we take this opportunity to fix coordinate systems to facilitate our calculations.

We identify the set of weights of $\fso(\bE)$ with $\bC^{m+n}$ with the usual scalar product $(x,y) = \sum_{i=1}^{m+n} x_iy_i$. Let $\eps_i$ be the $i$th standard basis vector of $\bC^{m+n}$. Our choice of simple roots for $\fso(\bE)$ is
\[
  \{\eps_i - \eps_{i+1} \mid i=1,\dots,m+n-1\} \cup \{\eps_{m+n}\}.
\]
Then the positive roots of $\fso(\bE)$ are of three kinds:
\begin{itemize}
\item $\eps_i$ for $i=1,\dots,m+n$,
\item $\eps_i - \eps_j$ for $1 \le i < j \le m+n$,
\item $\eps_i + \eps_j$ for $1 \le i < j \le m+n$.
\end{itemize}

The half-sum of the positive roots is
\[
  \rho = \frac12(2m+2n-1, 2m+2n-3, \dots, 3, 1).
\]
The Weyl group $W$ is the set of all signed permutations of size $m+n$ which acts on $\bC^{m+n}$ as signed permutation matrices. The dotted action of $W$ on $\bC^{m+n}$ is given by
\[
  w \bullet \lambda = w(\lambda+\rho)-\rho.
\]
Given a (positive) root $\alpha$, we let $s_\alpha \in W$ denote the corresponding reflection defined by
\[
  s_\alpha(\lambda) = \lambda - 2 \frac{(\alpha, \lambda)}{(\alpha, \alpha)} \alpha.
\]
We use the shorthand $\langle \alpha, \lambda \rangle = 2 \frac{(\alpha, \lambda)}{(\alpha, \alpha)}$. Note that $\langle \alpha, \lambda \rangle = (\alpha, \lambda)$ for the last 2 kinds of positive roots and $\langle \eps_i, \lambda \rangle = 2(\alpha, \lambda)$.

Given weights $\lambda, \mu$, say that a sequence of positive roots $(\gamma_1,\dots,\gamma_r)$ \defi{links $\lambda$ to $\mu$} if, setting $\lambda(i) = (s_{\gamma_i}\cdots s_{\gamma_1}) \bullet \lambda$ (and $\lambda(0)=\lambda$), we have
\begin{enumerate}
\item $\lambda(r) = \mu$, and
\item for each $i=1,\dots,r$, we have $\lambda(i) = \lambda(i-1) - n_i \gamma_i$ where $n_i = \langle \lambda(i-1) + \rho, \gamma_i \rangle$ is a positive integer.
\end{enumerate}
If there exists a sequence as above, then \defi{$\lambda$ is linked to $\mu$}.

We are going to use \cite{boe1985homomorphisms}. Note that the indexing of Verma modules used there is not the same as ours: in that reference, $M(\lambda)$ is what we would call $M(\lambda - \rho)$. Hence we have to use the dotted action rather than the usual action when acting on the weights indexing our Verma modules, cf. \cite[Definition 3.2]{boe1985homomorphisms}.

Finally, our set $\Lambda^+_I$ consists of sequences $(\alpha_1,\dots,\alpha_n, \beta_1,\dots,\beta_m)$ such that $\alpha_1 \ge \cdots \ge \alpha_n$, the differences between the $\alpha_i$ are all integers, $\beta_1 \ge \cdots \ge \beta_m \ge 0$, and $\beta_i \in \bZ_{\ge 0}$. As mentioned above, these are the weights whose restriction to $\fgl(n) \times \fso(2m+1)$ is the highest weight of a finite-dimensional representation.

Finally, if $\lambda$ is linked to $\mu$, then there is a homomorphism between parabolic Verma modules $M(\mu) \to M(\lambda)$ called the standard map. We record \cite[Theorem 3.3]{boe1985homomorphisms}:

\begin{theorem}[Boe] \label{thm:boe}
Given $\lambda, \mu \in \Lambda_I^+$ such that $\lambda$ is linked to $\mu$. Then the standard map $M(\mu)\to M(\lambda)$ is zero if and only if there is a sequence $(\gamma_1,\dots,\gamma_r)$ which links $\lambda$ to $\mu$ such that $\lambda(1) \notin \Lambda_I^+$.
\end{theorem}

\begin{proof}[Proof of Proposition~\ref{prop:main}]
  Let $\zeta$ denote the weight in \eqref{eqn:weight}. Set $\mu = s_{\eps_n} \bullet \zeta = \zeta - 2\eps_n$. Then $\zeta$ is linked to $\mu$ and we claim that the only sequence linking $\zeta$ to $\mu$ is the length 1 sequence $(\eps_n)$. If not, suppose $(\gamma_1,\dots,\gamma_r)$ links $\zeta$ to $\mu$ and that $r>1$. We have the following decomposition into simple roots
  \[
    2\eps_n = 2(\eps_n - \eps_{n+1}) + 2(\eps_{n+1} - \eps_{n+2}) + \cdots + 2(\eps_{n+m-1} - \eps_{n+m}) + 2\eps_{n+m}
  \]
Since $\zeta-\mu$ is a non-negative linear combination of the $\gamma_i$, we conclude that some $\gamma_i$ must be a positive root of the form $\eps_n \pm \eps_j$ where $j > n$. If $i$ is the smallest index where this happens, then the first $n$ entries of $\lambda(i-1)+\rho$ are all integers, but the last $m$ entries are all half-integers, so that $\langle \lambda(i-1) + \rho, \gamma_i \rangle$ is not an integer. This proves the claim.

Then Theorem~\ref{thm:boe} implies that the standard map $M(\mu)\to M(\zeta)$ is nonzero.  Next, by Remark~\ref{rmk:hw}, we have
\[
  M(\zeta) = \rU(\fg) \otimes_{\rU(\fp)} \bC_{m-\frac12}, \qquad M(\mu) = \rU(\fg) \otimes_{\rU(\fp)} (S^2(E) \otimes \bC_{m-\frac12}).
\]
Hence the image of the standard map is generated by $S^2(E)$, but there is exactly one copy of $S^2(E)$ appearing as a direct summand of $M(\zeta)$, so it coincides with the inclusion we are interested in.
\end{proof}

\subsection{The equivariant Hilbert function of $Z_{V,E}$} \label{sec:lowerbound}

To compute the character of the representation \( Z_{V,E} \) introduced earlier, we leverage several tools from representation theory, focusing on decompositions of module quotients and properties of certain induced representations. Specifically, we analyze the action of the Lie algebra \( \fso(\bE) \) on \( Z_{V,E} \) by examining submodules and highest-weight vectors.

Recall that $Z_{V,E}$ is defined to be the quotient of $M(\zeta)$ by the left $\rU(\fg_{V,E})$-module generated by $S^2 E$:
\[
  Z_{V,E} = M(\zeta) / (S^2E) \cong \rU( \fg_{V,E} ) / (S^2 E).
\]
Quotienting by the left ideal generated by $\fso(\bE)_2=\bigwedge^2 E$ yields
\[
  S^\bullet(V \otimes E)/(S^2 E).
\]
where now $(S^2 E)$ is the ideal generated by $S^2 E$ in the algebra $S^\bullet(V \otimes E)$.

We will denote the irreducible representation of $\bO(V)$ of highest weight $\lambda$ by $\bS_{[\lambda]}(V)$. As in the $\Sp\bO$ case, these are indexed by admissible partitions, i.e., those that satisfy $\lambda_1^T + \lambda_2^T \le \dim V$. Some basic information about their construction can be found in \cite[\S 19.5]{fulton2013representation}.

\begin{lemma} \label{lem:lw-quotient}
  The above quotient has a decomposition 
\[
    S^\bullet(V \otimes E)/(S^2 E) \cong \bigoplus_{\lambda} \bS_{[\lambda]}(V) \otimes \bS_\lambda(E)
  \]
  where the sum is over all admissible partitions $\lambda$.
\end{lemma}

We should add the requirement that $\ell(\lambda)\le n$, but $\bS_\lambda(E)=0$ if $\ell(\lambda)>n$ so for the purposes of this formulation, that requirement is redundant.

\begin{proof}
  To see this, we can use \cite[Lemma 4.21]{sam2013homology} (more specifically, its proof) which gives an $(\bO(V) \times \GL(E))$-equivariant decomposition 
  \[
    S^\bullet(V \otimes E) = \bigoplus_\lambda \bS_{[\lambda]}(V) \otimes M_\lambda,
  \]
  where the sum is over all admissible partitions $\lambda$ such that $\ell(\lambda) \le n$, and $M_\lambda$ is a $\GL(E)$-equivariant $S^\bullet (S^2 E)$-module which is minimally generated by $\bS_\lambda(E)$ (here the action of $S^2 E$ is via the inclusion $S^2 E \subset S^2(V \otimes E)$). In particular, quotienting by the ideal generated by $S^2 E$ has the effect of replacing $M_\lambda$ by $\bS_\lambda(E)$, which implies our desired result.
\end{proof}

As an immediate consequence, if we consider the action of $\fso(E \oplus E^*)$ on $Z_{V,E}$, it contains a nonzero weight vector of weight $\zeta' - \lambda^{\rm op}$ for each admissible partition $\lambda$ with $\ell(\lambda) \le n$. In particular, the following direct sum must be a quotient of $Z_{V,E}$:
\[
  \bigoplus_{\lambda} \bS_{[\lambda]}(V) \otimes L_{\fso(E \oplus E^*)}(\zeta' - \lambda^{\rm op})
\]
where again the sum is over all admissible partitions $\lambda$ such that $\ell(\lambda) \le n$.

We let $M_{\fso(E \oplus E^*)}(\zeta'-\lambda^{\rm op})$ denote the corresponding parabolic Verma module with respect to the decomposition \eqref{eqn:parabolic}. More precisely, 
\[
  M_{\fso(E \oplus E^*)} (\zeta' - \lambda^{\rm op}) = {\rm Ind}_{\fk \oplus \fn_+}^{\fso(E \oplus E^*)} L_{\fgl(E)}(\zeta'-\lambda^{\rm op})
\]
where $L_{\fgl(E)}(\zeta'-\lambda^{\rm op})$ is the irreducible $\fgl(E)$-representation with highest weight $(\frac12-m-\lambda_n,\dots,\frac12-m-\lambda_1)$. Then as an $\fso(E \oplus E^*)$-representation, we have
\[
  \rU(\fg_{V,E}) \cong {\rm Ind}^{\fso(E \oplus E^*)}_{\fk \oplus \fn_+} (S^\bullet (V \otimes E)) \cong \bigoplus_\lambda \bS_{\lambda}(V) \otimes M_{\fso(E\oplus E^*)}(\zeta' - \lambda^{\rm op}).
\]
By standard Lie theory, $L_{\fso(E \oplus E^*)}(\zeta'-\lambda^{\rm op})$ is naturally a quotient of $M_{\fso(E \oplus E^*)}(\zeta' - \lambda^{\op})$, and we define
  \[
    N_\lambda = \ker(M_{\fso(E \oplus E^*)}(\zeta' - \lambda^\op) \to L_{\fso(E \oplus E^*)}(\zeta' - \lambda^\op)).
  \]
  
  \begin{proposition}
    Let $J$ be the image of $\bigoplus_\lambda \bS_{[\lambda]}(V) \otimes N_\lambda$ in $Z_{V,E}$ under the map $\rU(\fg_{V,E}) \to Z_{V,E}$. Then $J$ is a $\rU(\fso(\bE))$-submodule of $Z_{V,E}$.
  \end{proposition}

  \begin{proof}
    Let $\lambda$ be an admissible partition. Let $N^\circ_\lambda$ be the space of minimal generators of $N_\lambda$ as an $S^\bullet(\bigwedge^2 E)$-module, so that as vector spaces we have an isomorphism $N^\circ_\lambda \to N_\lambda\otimes_{S^\bullet(\bigwedge^2 E)} \bC$. By Corollary~\ref{cor:Mlambda-reldeg} and Lemma~\ref{lem:SwedgeModuleIso}, there is an integer $d$  such that $N^\circ_\lambda \subset \bS_\lambda E \otimes S^d(\bigwedge^2 E)$. Finally, let $P_\lambda$ be the image of $\bS_{[\lambda]}(V) \otimes N^\circ_\lambda$ in $Z_{V,E}$.

    Then $P_\lambda$ is annihilated by $\bigwedge^2 E^* \subset \fso(\bE)$. Since $\bigwedge^2 E^*$ and $V \otimes E^*$ commute, the image of $P_\lambda$ under $V \otimes E^*$ is also annihilated by $\bigwedge^2 E^*$. Note that $V \otimes E^*$ takes $P_\lambda$ into a sum of terms which are of one of two forms:
    \begin{itemize}
    \item $\bS_{[\mu]}(V) \otimes \bS_\mu E \otimes S^{d-1}(\bigwedge^2 E)$ where $\mu$ is an admissible partition such that $|\mu|=|\lambda|+1$, or 
    \item $\bS_{[\nu]}(V) \otimes \bS_\nu E \otimes S^d(\bigwedge^2 E)$ where $\nu$ is an admissible partition such that $|\nu|=|\lambda|-1$. 
    \end{itemize}
    If $d>1$, then these terms are still in positive degree, and any elements annihilated by $\bigwedge^2 E^*$ in positive degree must be contained in $N_\mu$ and hence contained in $J$. If $d=1$, then the second type of terms are still in positive degree, and we claim that the first type do not exist. To see this, consider Corollary~\ref{cor:Mlambda-reldeg} and Remark~\ref{rmk:iso-mult}, which together imply that $N_\lambda^\circ$ consists of Schur functors $\bS_{\alpha}(E)$ such that $\alpha_1^T + \alpha_2^T \ge 2m+3$. But then $\bS_\mu E \subseteq \bS_\alpha E \otimes E^*$, which means that $\mu_1^T + \mu_2^T \ge 2m+2$, and hence $\mu$ is not an admissible partition.
    
    Hence we have shown that $J$ is closed under $V \otimes E^*$ and $\bigwedge^2 E^*$. By PBW, the $\fso(\bE)$-submodule generated by $J$ is then the same as the $\fp$-submodule that it generates, but it is already closed under $\fp$.
  \end{proof}

  \begin{corollary}\label{cor:irredAndTrivialJ}
    $Z_{V,E}$ is an irreducible $\fso(\bE)$-module and $J=0$.
  \end{corollary}

  \begin{proof}
If $\nu$ is the weight of any nonzero highest weight vector in any subquotient of $M(\zeta)$, then there exists $w \in W$ such that $\nu = w \bullet \zeta$. This follows from a combination of a few things: first $M(\zeta)$ is indecomposable since every submodule has a weight decomposition and any submodule containing the weight space for $\zeta$ must be the whole space. Second, we can use central characters of $\rU(\fso(\bE))$ to decompose general modules, so there is a unique central character acting on $M(\zeta)$ \cite[\S 1.12]{humphreys2008representations}. Next, the central character for the irreducible representation $L(\nu)$ with highest weight $\nu$ is the same as that of $L(\lambda)$ if and only if $\nu = w \bullet \lambda$ \cite[\S 1.10]{humphreys2008representations}.

Next, let us examine which weights of the form $w \bullet \zeta$ appear in $M(\zeta)$. First, the first $n$ entries of $\zeta$ are half-integers (i.e., elements of the coset $\frac12 + \bZ$) while the last $m$ entries are integers. If $w$ swaps any of the first $n$ entries with some of the last $m$ entries, then the resulting weight will have half-integer entries in the last $m$ entries. But $M(\zeta)$ has no such weights since they are all of the form $\zeta$ plus some integer sequence. So $w$ must permute (possibly with signs) the first $n$ entries amongst themselves and the same for the last $m$ entries. However, the last $m$ entries of $\zeta$ are all $0$, so acting on this in any non-trivial way will give negative entries, so cannot belong to $\Lambda_I^+$. Hence all of the nontrivial action must occur in the first $n$ entries.

We claim that the only possible weights in $\Lambda_I^+$ that can result from acting on the first $n$ entries are of the form $(-\beta_n,\dots, -\beta_1,0,\dots,0) + \zeta$ where $\bS_\beta$ appears as a summand of $\bigwedge^\bullet (S^2)$. This can be proven directly, but here is an alternative explanation. First, if we focus on the first $n$ entries only, then we are asking what sequences arise as $w(\zeta + \rho) - \zeta - \rho$ where $w$ is a signed permutation. But the first $n$ entries of $\zeta + \rho$ is the sequence $(n,n-1,\dots,1)$, so by \cite[\S 2]{heisenberg} (take $n$ and $k$ to be $0$ and $n$ in the notation there) this amounts to computing the representations in the Koszul complex on $S^2 E$, i.e., the exterior algebra $\bigwedge^\bullet(S^2 E)$. In particular, $|\beta|$ is even and the action of $\fso(V)$ on this representation is trivial.

Using the PBW degeneration \cite{braverman1996poincare}, we get an upper bound for the equivariant Hilbert series of $Z_{V,E}$ from $S^\bullet(E \otimes V)/(S^2 E) \otimes S^\bullet(\bigwedge^2 E)$. Note that $\fso(V)$-invariants in even degree are also $\bO(V)$-invariants. By Lemma~\ref{lem:lw-quotient}, the quotient $S^\bullet(E \otimes V)/(S^2 E)$ has no non-constant $\bO(V)$-invariants. 

In particular, the $\fso(V)$-invariants which can potentially be highest weight vectors only appear as direct summands of $S^\bullet(\bigwedge^2 E)$, i.e., with the Schur functors $\bS_{(2\alpha)^T}(E)$. The highest weights of the irreducible representations appearing in both $\bigwedge^\bullet (S^2 E)$ and $S^\bullet(\bigwedge^2 E)$ (in terms of irreducible direct summands) consists of the rectangles of the form $s \times (s+1)$, where $s$ is even. 

Since $J$ is a submodule, if $J \ne 0$, then it must contain a nonzero highest weight vector. In particular, $J$ must contain $\bS_{s \times (s+1)}(E)$ for some $s$, so that $\bS_{s \times (s+1)}(E) \subset N_0$. Now apply the reasoning from the first paragraph to $\fso(E \oplus E^*)$ acting on $M(\zeta')$: 
\[
  (0,\dots,0,\underbrace{-s-1,\dots,-s-1}_s) + \zeta'
\]
is in the shifted Weyl group orbit of $\zeta'$ (under the Weyl group of $\fso(E\oplus E^*)$; we are in type D, so $\rho = (n-1,\dots,1,0)$). However, if $-\lambda^\op + \zeta'$ is in the shifted Weyl group orbit of $\zeta'$ and $\lambda \ne 0$, then we claim that we must have $\ell(\lambda) = \lambda_1 + 2m$: to get a sequence of the form $-\lambda^\op + \zeta'$ starting from $\zeta'$, we pick a subsequence $n-s_1-(2m-1)/2 > \cdots > n-s_r - (2m-1)/2$ (with $r$ even) of $\zeta' + \rho$, negate each term of the subsequence, and move this subsequence to the end of the tuple and subtract $\rho$. In that case, none of the entries to the left of $n-s_1-(2m-1)/2$ are affected, so $\ell(\lambda) = n-s_1+1$. Furthermore, the last entry of $\zeta'+\rho$ becomes $-n+s_1+(2m-1)/2$, which means that $\lambda_1 = n-s_1-2m+1$.

Hence $\lambda$ cannot be a rectangle of shape $s \times (s+1)$, so we conclude that $J = 0$. This same argument also shows that $Z_{V,E}$ contains no non-constant highest weight vectors, so it is irreducible.
  \end{proof}

  \begin{corollary}
    \begin{enumerate}
    \item We have an isomorphism of $\bO(V) \times \fso(E \oplus E^*)$-representations
    \[
      Z_{V,E} \cong \bigoplus_\lambda \bS_{[\lambda]}(V) \otimes L_{\fso(E \oplus E^*)}(\zeta' - \lambda^{\rm op})
    \]
    where the sum is over all partitions $\lambda$ such that $\lambda_1^T + \lambda_2^T \le 2m+1$ and $\lambda_1^T \le n$.
  \item The characters of $Z_{V,E}$ and $S^\bullet(U \otimes E)$ agree, where $U = \bC^{2m|1}$ (via the isomorphism of character rings $\text{K} (\bO (2m+1)) \cong {\rm K} (\Sp \bO (2m | 1))$).
  \item Using the grading $\deg(E \otimes V) = 1$ and $\deg(\bigwedge^2 E)=2$ for $\rU(\fg_{E,V})$, the Hilbert series of $Z_{V,E}$ is $\left( \frac{1+t}{(1-t)^{2m}} \right)^n$.
  \end{enumerate}
\end{corollary}

\begin{proof}
  (1) By definition, the $\rU (\fso (\bE))$-submodule $J$ is the kernel of the induced surjection $Z_{V,E} \twoheadrightarrow \bigoplus_\lambda \bS_{[\lambda]}(V) \otimes L_{\fso(E \oplus E^*)}(\zeta' - \lambda^{\rm op})$. By Corollary \ref{cor:irredAndTrivialJ}, this surjection must in fact be an isomorphism.

  (2) Combine (1) with Proposition~\ref{prop:spo-howe}.

  (3) As a graded vector space, we have $S^\bullet(U \otimes E) \cong S^\bullet(\bC^{2m} \otimes E) \otimes \bigwedge^\bullet(E)$, so its Hilbert series is $(1-t)^{-2nm} (1+t)^n$, so we finish by using (2).
\end{proof}

\section{Loose ends}

\subsection{Quadric Schur modules as irreducible representations}\label{subsec:quadricAsIrred}

Set $R \coloneqq \rU (\so (\bE))^{\GL(E)}$, the subalgebra of $\GL (E)$-invariants of $\rU (\so (\bE))$ and let $A = S^\bullet(V)/(q)$ be the quadric hypersurface associated to $(V,\beta)$.

\begin{theorem}\label{thm:irredQuadricThm}
   \begin{enumerate}
       \item For every $\lambda$ with $\ell (\lambda) \leq \dim E$ and $\lambda_i \leq 1$ for $i \geq \dim V -1$, the quadric Schur module $\bS_\lambda^A$ is an irreducible $R$-module.
       \item The set of $\bS_\lambda^A$ with $\lambda$ as above constitutes a set of non-isomorphic $R$-modules.
   \end{enumerate} 
\end{theorem}

This is an immediate consequence of \cite[Theorem 5.8]{cheng2012dualities}.

\begin{remark}
  As alluded to in the introduction, we do not understand how to describe the algebra $R$ in a more explicit fashion. It contains the subalgebra generated by $\fso(V)$ and is strictly larger. Also, we are actually interested in the image of $R$ in the endomorphism algebra of $Z_{V,E}$.

  The algebra $R$ has some interesting properties as the dimension of $E$ grows. For instance, let $\lambda$ be any partition of length $\dim V-1$; then there are $\bO (V)$-equivariant isomorphisms
    $$\bS^A_{(\lambda, 1^i)} \cong \bS^A_{(\lambda , 1^{i-2})} \quad \text{for all} \ i \geq 2.$$
    On the other hand, if $\dim E \geq \dim V + i$, then Theorem \ref{thm:irredQuadricThm} tells us that $\bS^A_{(\lambda, 1^i)}$ and $\bS^A_{(\lambda , 1^{i-2})}$ are non-isomorphic $R$-modules.
\end{remark}

For the remainder of this section, we explicitly construct a subalgebra of $R$ that is strictly larger than $\rU (\so (V))$.

In the following, we employ notation from \cite[\S 20.1]{fulton2013representation} on Clifford algebras. In particular, the Clifford algebra $\Cliff({\bf E})$ is the tensor algebra of ${\bf E}$ modulo the relations
\[
  v w  + wv = 2\beta(v,w).
\]
We have a map $\psi \colon \bigwedge^2 {\bf E} \to  \Cliff({\bf E})$ given by
\[
  \psi(v \wedge w) = \frac12 (vw - wv) = vw - \beta(v,w).
\]
Then $\psi$ is a homomorphism of Lie algebras, where $\bigwedge^2 {\bf E}$ has the bracket (coming from its isomorphism with $\so({\bf E})$)
\[
  [a\wedge b, c \wedge d] = 2 (\beta(b,c) a \wedge d + \beta(a,d) b \wedge c - \beta(b,d) a \wedge c - \beta(a,c) b \wedge d),
\]
and $\Cliff({\bf E})$ is given the usual commutator bracket from its associative algebra structure.

In particular, $\psi$ extends to a homomorphism of associative algebras
\[
  \psi \colon \uU(\so({\bf E})) \to \Cliff({\bf E}).
\]

Recall that $\Cliff({\bf E}) \cong \bigwedge^\bullet {\bf E}$ as an $\so({\bf E})$-representation. For all $d,e \ge 0$, we have a $\GL(E)$-invariant subspace
\begin{equation}\label{eqn:extAlgInc}
    \bigwedge^{2d} V \subset \bigwedge^e E^* \otimes \bigwedge^{2d} V \otimes \bigwedge^e E \subset \bigwedge^{2d+2e}(E^* \oplus V \oplus E).
\end{equation}
Let us focus on the case $d=e=1$. Choose a basis $e_1,\dots,e_n$ for $E$ and let $e_{-1},\dots,e_{-n}$ be the dual basis of $E^*$. Let $v_1,\dots,v_{2m+1}$ be an orthonormal basis for $V$, so that $\psi(v_i\wedge v_j) = v_iv_j$. Then the embedding above is given by the formula
\[
  v_i \wedge v_j \mapsto \sum_{\alpha} e_{-\alpha} v_i v_j e_\alpha.
\]
To lift this to a map $\bigwedge^2 V \to \uU(\so({\bf E}))$, treat $e_{-\alpha} v_i$ as an element of $E^* \otimes V$ and similarly treat $v_j e_{\alpha}$ as an element of $V \otimes E$. Recall that we have the decomposition
\[
  \so({\bf E}) \cong \bigwedge^2 E^* \oplus (E^* \otimes V) \oplus (\so(V) \times \gl(E)) \oplus (V \otimes E) \oplus \bigwedge^2 E.
\]
So define
\[
  \phi_{i,j} = \sum_\alpha (e_{-\alpha} v_i) (v_j e_\alpha) \in \uU(\so({\bf E})).
\]
Let us compute the action of $\phi_{i,j}$ on 
\[
  \bigwedge^2 E \subset S^\bullet (\bigwedge^2 E) \subset S^\bullet (V \otimes E) \otimes S^\bullet (\bigwedge^2 E)\cong M(\zeta),
\]
where we recall that $M(\zeta) = {\rm Ind}_\fp^{\fso(\bE)} \bC_\zeta$ denotes the parabolic Verma module for the weight $\zeta$ defined in \eqref{eqn:weight}. Employing the notation $e_{12} = e_1 \wedge e_2$, we have
\begin{align*}
  \phi_{i,j} \cdot  e_{12} &= \sum_{\alpha} (e_{-\alpha} v_i) (v_je_\alpha) (e_{12}) \\
                    &= \sum_{\alpha} (v_je_\alpha)(e_{-\alpha} v_i) (e_{12}) + 2n (v_iv_j)(e_{12})\\
                    &= \sum_{\alpha} (v_je_\alpha)(e_{12}) (e_{-\alpha} v_i) - 2(v_je_1)(v_ie_2) + 2(v_je_2)(v_ie_1)  + 2n (v_iv_j)(e_{12}).
\end{align*}
To simplify: by construction $e_{-\alpha} v_i$ acts trivially on $\bC_\zeta$, so the summation becomes 0. For the last term: $(v_iv_j)$ and $(e_{12})$ commute, and $v_iv_j$ also acts trivially on $\bC_\zeta$, so the term $2n (v_iv_j)(e_{12}) $ is also $0$. The middle two terms are precisely a $2 \times 2$ minor in $\bigwedge^2 V \otimes \bigwedge^2 E$ (up to a factor of $2$).

Recall that the multiplicity space of $\bigwedge^2 E$ in the Verma module $M(\zeta)$ is $\bigwedge^2 V \oplus \bC$, in which case the above shows that $\phi_{i,j}$ induces a nonzero action from $\bC$ to $\bigwedge^2 V$.

\subsection{Orthosymplectic analogue} \label{ss:orthosymp-version}

Going back to the table in the introduction, there is a parallel of the construction in this paper for the supergroup $\Sp\bO(2m|1) = \Sp\bO(U)$. More precisely, under the isomorphism $\rK(\Sp\bO(U)) \cong \rK(\bO(2m+1))$, the classes of exterior powers $[\bigwedge^i(\bC^{2m+1})]$ correspond to the classes of the cokernels $\bigwedge^{i-2}(U) \to \bigwedge^i(U)$.

In particular, the class of the exterior algebra $\bigwedge^\bullet(\bC^{2m+1})$ corresponds to the class of the quotient
\[
  C = \bigwedge^\bullet(U)/(\omega)
\]
where $\omega$ is the quadratic element given by the symplectic form. In coordinates, if we choose a basis $x_1,\dots,x_{2m}$ for $\bC^{2m}$ and $x_0$ for $\bC^1$, then
\[
  \omega = \sum_{i=1}^m x_{2i-1} \wedge x_{2i} + x_0^2.
\]
The algebra $C$ is a Koszul algebra, as can be seen directly by considering the deformed quadric $\omega_t \coloneqq \sum_{i=1}^m t \cdot x_{2i-1} \wedge x_{2i} + x_0^2$: when $t= 0$ the quotient $\bigwedge^\bullet (U) / (\omega_0)$ decomposes as the tensor product of Koszul algebras $\bigwedge^\bullet (U_0) \otimes S^\bullet (U_1) / (x_0^2)$, and is thus itself Koszul. Flatness thus forces $C$ to be Koszul.

In the language of \cite{sam2024total}, the exterior algebra on $\bC^{2m+1}$ is $\bO(2m+1)$-totally positive. We immediately deduce the following:

\begin{proposition}
  The Koszul algebra $C$ is $\Sp\bO(2m|1)$-totally positive.
\end{proposition}

We can think of $C$ as the analogue of the quadric hypersurface for the group $\Sp\bO(2m|1)$, so this provides an orthosymplectic parallel to what we have been studying. We do not know how to prove the analogue of our current main theorem about ``good functorial properties'' but it is easy to formulate a conjecture, so we will conclude the paper by doing so.

Namely, let $E$ be a vector space and define
\[
  \fg'_{U,E} = (U \otimes E) \oplus S^2 E.
\]
This is naturally a superspace with even part $(U_0 \otimes E) \oplus S^2 E$ and odd part $U_1 \otimes E$. We define a Lie superbracket as follows: the elements of $S^2 E$ are central and otherwise for $u \otimes e, u' \otimes e' \in U \otimes E$, we define
\[
  [u \otimes e, u' \otimes e'] = \omega(u,u') e e' \in S^2 E.
\]

We remark that $\fg'_{U,E}$ has a grading where $\deg(U \otimes E)=1$ and $\deg(S^2E)=2$ which is compatible with this bracket, but this grading {\it does not} collapse to the $\bZ/2$-grading given by the Lie superalgebra structure. We will use this grading to define Hilbert series below.

Using $\omega$ and the isomorphism $U \cong U^*$, we also have a subspace
\[
  \bigwedge^2 E \subset \bigwedge^2 U \otimes \bigwedge^2 E \subset S^2(U \otimes E),
\]
and we can take the left ideal $I'$ generated by $\bigwedge^2 E$ in the enveloping algebra $\rU(\fg'_{U,E})$. Finally, we define
\[
  Z'_{U,E} = \rU(\fg'_{U,E}) / I'.
\]

\begin{conjecture}
  Let $n = \dim E$. Then we have an $\Sp\bO(U)$-equivariant isomorphism
  \[
    Z'_{U,E} \cong C^{\otimes n}.
  \]
  In particular, we have the following identity of Hilbert series:
  \[
    \rH\rS(Z'_{U,E}) = \rH\rS(C)^n.
  \]
\end{conjecture}

We can realize $\fg'_{U,E}$ as the nilpotent radical of a parabolic subalgebra of the Lie superalgebra $\fsp\fo(2m+2n|1)$. The arguments from this article {\it should} extend to this case, but we are not aware of sources for the relevant Lie theory facts.

\bibliography{biblio}
      
\end{document}